\documentclass{article}
\usepackage[utf8]{inputenc}
\usepackage[english]{babel}
 
\usepackage[dvipsnames]{xcolor}
 
\usepackage[nottoc]{tocbibind}
\usepackage{amsmath}
\usepackage{cases}
\usepackage{amsfonts}
\usepackage{amssymb}
\usepackage{array}
\usepackage{mathtools}
\usepackage{tikz}
\usepackage{tikz-cd}
\usepackage{ragged2e}
\usepackage{amsthm}
\usepackage{tabularx}
\usepackage{centernot}
\usepackage{faktor}
\usepackage{rotating}
\usepackage{commath}
\usepackage{comment}
\usepackage{mathabx}
\usepackage{fge}
\usepackage[parfill]{parskip}
\usepackage[shortlabels]{enumitem}
\usepackage{xifthen}
\usepackage{etoolbox}
\usepackage{verbatim}
\usepackage{tikzducks}
\usepackage{ytableau}
\usepackage{subfig}
\usepackage{slashed}
\usepackage{xfrac}
\usepackage[title]{appendix}
\usepackage{multirow}
\usepackage{textcomp}
\usepackage[hyphens]{url}

\usepackage[margin=1.15in]{geometry}

\usepackage{pdflscape}
\usepackage{afterpage}
\usepackage{capt-of}

\usepackage{lipsum}
\usepackage{hyperref}

\hypersetup{
    colorlinks=true,
    linkcolor=BrickRed,
    filecolor=Apricot,      
    urlcolor=Rhodamine,
    citecolor=ForestGreen,
    anchorcolor=Fuchsia,
}

\DeclareFontFamily{U}{MnSymbolC}{}
\DeclareSymbolFont{MnSyC}{U}{MnSymbolC}{m}{n}
\DeclareFontShape{U}{MnSymbolC}{m}{n}{
    <-6>  MnSymbolC5
   <6-7>  MnSymbolC6
   <7-8>  MnSymbolC7
   <8-9>  MnSymbolC8
   <9-10> MnSymbolC9
  <10-12> MnSymbolC10
  <12->   MnSymbolC12}{}
\DeclareMathSymbol{\hook}{\mathbin}{MnSyC}{'270}

\newcommand{\bb}[1]{
    \mathbb{#1}
}
\renewcommand{\cal}[1]{
    \mathcal{#1}
}

\newcommand{\rarr}{
    \rightarrow
}

\renewcommand{\:}{
    \colon
}

\renewcommand{\del}[1]{
    \partial #1
}

\newcommand{\wtilde}[1]{
    \widetilde{#1}
}

\newcommand{\Ric}{
    \text{Ric}
}

\newcommand{\vol}{
    \text{\normalfont{vol}}
}
\renewcommand{\Re}{
    \text{\normalfont{Re}}
}
\renewcommand{\Im}{
    \text{\normalfont{Im}}
}
\newcommand{\id}{
    \text{\normalfont{id}}
}
\newcommand{\CY}{
    \text{\normalfont{CY}}
}

\newtheorem{innercustomgeneric}{\customgenericname}
\providecommand{\customgenericname}{}
\newcommand{\newcustomtheorem}[2]{%
  \newenvironment{#1}[1]
  {%
   \renewcommand\customgenericname{#2}%
   \renewcommand\theinnercustomgeneric{##1}%
   \innercustomgeneric
  }
  {\endinnercustomgeneric}
}

\makeatletter
\providecommand{\subtitle}[1]{
  \apptocmd{\@title}{\par {\large #1 \par}}{}{}
}
\makeatother

\makeatletter
\def\mathcolor#1#{\@mathcolor{#1}}
\def\@mathcolor#1#2#3{%
  \protect\leavevmode
  \begingroup
    \color#1{#2}#3%
  \endgroup
}
\makeatother

\newcustomtheorem{customthm}{Theorem}
\newcustomtheorem{customlem}{Lemma}
\newcustomtheorem{customprop}{Proposition}

\theoremstyle{plain}
\newtheorem{thm}{Theorem}[section]
\newtheorem{lem}[thm]{Lemma}

\theoremstyle{definition}
\newtheorem{defn}[thm]{Definition}

\newtheorem{rmk}[thm]{Remark}

\theoremstyle{remark}

\title{Flows of $G_2$-Structures Associated to Calabi--Yau Manifolds}
\author{
  S\'ebastien Picard \\
\texttt{spicard@math.ubc.ca}
  \and
  Caleb Suan \\
\texttt{calebkw@math.ubc.ca}
}

\date{}

\numberwithin{equation}{section}

\begin{document}

\maketitle


\begin{abstract}
\label{abstract}
We establish a correspondence between a parabolic complex Monge--Amp\`ere equation and the $G_2$-Laplacian flow for initial data produced from a K\"ahler metric on a complex $2$- or $3$-fold. By applying estimate for the complex Monge--Amp\`ere equation, we show that for this class of initial data the $G_2$-Laplacian flow exists for all time and converges to a torsion-free $G_2$-structure induced by a K\"ahler Ricci-flat metric. Similar results are obtained for the $G_2$-Laplacian coflow, and in this case the coflow is related to the K\"{a}hler--Ricci flow.
  \end{abstract}

\tableofcontents

\section{Introduction}
\label{sect-intro}

Since the works of Eells--Sampson \cite{ES64} and Hamilton \cite{Ham82}, geometric flows have appeared in a variety of settings in differential geometry, including minimal surface theory \cite{Hui84}, Yang--Mills theory \cite{Don85}, K\"ahler geometry \cite{Cao85}, $G_2$ geometry \cite{Bry05,DGK21,Kar09,Lot20}, generalized geometry \cite{GFS21}, and string theory \cite{FPPZ21a,Pho20,PPZ18CAG,PPZ18}. In this paper, we will study the interaction between geometric flows in $G_2$ geometry and flows in K\"ahler geometry via dimension reduction. The two most prominent flows in $G_2$ geometry are the Laplacian flow \cite{Bry05} and the Laplacian coflow \cite{KMT12}. We will study these flows for initial data coming from K\"ahler geometry, and establish a link between them and the MA$^{\frac{1}{3}}$ and K\"ahler--Ricci flows.

Let $X^{2n}$ be a compact complex manifold of complex dimension $n = 2$ or $3$ with K\"ahler metric $\omega$ and holomorphic volume form $\Omega$; we take this as the definition of a Calabi--Yau $n$-fold. The setting for $G_2$ geometry is a manifold $M^7$ of dimension $7$ with a positive $3$-form $\varphi$. An example of a positive 3-form $\varphi$ comes from considering the product $M^7 = T^3 \times X^4$ and setting
\begin{equation} 
\label{eqn-varphi-flow-dim-2}
    \varphi = - dr^1 \wedge dr^2 \wedge dr^3 + dr^1 \wedge \omega + dr^2 \wedge \Re (\Omega) + dr^3 \wedge \Im (\Omega)
\end{equation}
where $r^1, r^2$, and $r^3$ are the angle coordinates on $T^3$. Alternatively, when $n = 3$, we may consider the product $M^7 = S^1 \times X^6$ and set
\begin{equation} 
\label{eqn-varphi-flow-dim-3}
    \varphi = \Re (\Omega) - dr \wedge \omega
\end{equation}
where $r$ is the angle coordinate on $S^1$. These $3$-forms define torsion-free $G_2$-structures when $\omega=\omega_{\CY}$ is taken to be a K\"ahler Ricci-flat metric, but for an arbitrary K\"ahler metric $\omega$ the forms $\varphi$ are closed but not coclosed.

The study of flows of $G_2$ structures with partial integrability conditions imposed on ans\"{a}tze of the form \eqref{eqn-varphi-flow-dim-2}, \eqref{eqn-varphi-flow-dim-3} can be found in e.g. \cite{FR20,FY18, KMT12,LSES22}. Compared to these works, we work with K\"ahler structures, and our motivating question is whether the $G_2$-Laplacian flow on $M^7$ creates a flow of K\"ahler geometry on $X^{2n}$ which takes us in the limit to a K\"ahler Ricci-flat metric. This is similar in spirit to \cite{FPPZ21a,PPZ19}, where flows in symplectic/non-K\"ahler geometry become complex Monge--Amp\`ere flows in K\"ahler geometry for integrable initial data.
  
We now make this discussion more precise, with the relevant definitions to be found in \S \ref{sect-G2-Structs}. Start the $G_2$-Laplacian flow
\begin{equation}
\label{eqn-Laplacian-flow}
    \frac{d \varphi_t}{dt} = \Delta_{d_t} \varphi_t
\end{equation}
on $M^7 = T^3 \times X^4$ with initial data \eqref{eqn-varphi-flow-dim-2} or on $M^7 = S^1 \times X^6$ with initial data \eqref{eqn-varphi-flow-dim-3}. We will show that with this initial condition, the $G_2$-Laplacian flow on $M^7$ is equivalent to a complex Monge--Amp\`ere flow on $X^{2n}$, in the sense that the solution evolves as
\begin{equation}
\label{eqn-ansatz-varphi-flow}
    \varphi_t = 
    \begin{cases}
        - dr^1 \wedge dr^2 \wedge dr^3 + dr^1 \wedge \omega_t \\
        \qquad + dr^2 \wedge \Re (\Omega_t) + dr^3 \wedge \Im (\Omega_t) &\text{ on } T^3 \times X^4, \\[5pt]
        \Re (\Omega_t) - dr \wedge \omega_t &\text{ on } S^1 \times X^6,
    \end{cases}
\end{equation}
with $(\omega_t,\Omega_t) = (\Theta_t^* \wtilde{\omega}_t, \Theta_t^* \Omega)$, where $\wtilde{\omega}_t = \omega + i\partial \overline{\partial} u_t > 0 $, and $u_t$ solves
\begin{equation}
\label{eqn-MA1/3-flow}
    \frac{du_t}{dt} = 6K \Big(e^{-2 \log |\Omega|_{\omega}} \frac{\det (\omega + i \partial \overline{\partial} u_t)}{\det \omega} \Big)^\frac{1}{3}, \quad u_0 = 0,
\end{equation}
on $X^{2n}$ for the constant $K = 2^{\frac{n}{3}}$ and $\Theta_t$ is a family of diffeomorphisms determined by $\wtilde{\omega}_t$ with $\Theta_0= \id_{X^{2n}}$ (see \eqref{eqn-Theta_t-flow} for the definition of $\Theta_t$). 

\begin{rmk}
  The complex Monge--Amp\`ere flow \eqref{eqn-MA1/3-flow} is not the K\"ahler-Ricci flow \cite{Cao85}. This nonlinear equation with the exponent $\frac{1}{3}$ replaced by $-1$ has appeared in various setups \cite{CK13,CHT22,FLM11,FPPZ21a,FP21}, while the exponent $+1$ was derived from the Type IIB flow in \cite{PPZ19}. We will refer to equation \eqref{eqn-MA1/3-flow} as the MA$^{\frac{1}{3}}$ flow.
\end{rmk}

\begin{rmk}
  The power of $\frac{1}{3}$ is interesting from the point of view of fully nonlinear PDEs, as the general theory \cite{CNS3,Guan14,PhongTo,Sze18} assumes a concave operator, but $A \mapsto (\det A)^\frac{1}{3}$ on positive-definite matrices $A$ is not concave except for $3 \times 3$ matrices.
\end{rmk}

The Monge--Amp\`ere equation on a compact complex manifold with zero first Chern class was solved by S.-T. Yau in \cite{Yau78}. Since then, the analysis of the complex Monge--Amp\`ere equation has been pursued in various directions, including e.g. \cite{BT82,BBGZ13,EGZ09,Kolo98,PSS12,TY90,TW10}. A theory of parabolic complex Monge--Amp\`ere equations that are not necessarily concave was developed in joint work of the first-named author with X.-W. Zhang \cite{PZ19} (generalizing estimates obtained in previous work with D.H. Phong \cite{PPZ19}). The MA$^{\frac{1}{3}}$ flow \eqref{eqn-MA1/3-flow} fits into that framework, and as an application of the complex Monge--Amp\`ere estimates in \cite{PZ19} we obtain the following theorem on the $G_2$-Laplacian flow.

\begin{thm}
\label{thm-Laplacian-flow-MA1/3}
Let $X^{2n}$ be a compact K\"ahler Calabi--Yau $n$-fold where $n = 2$ or $3$ with K\"ahler metric $\omega$ and holomorphic volume form $\Omega$. Start the $G_2$-Laplacian flow with initial data
\[
\varphi =  - dr^1 \wedge dr^2 \wedge dr^3 + dr^1 \wedge \omega + dr^2 \wedge \Re (\Omega) + dr^3 \wedge \Im (\Omega) \quad \text{ on } T^3 \times X^4 
\]
or
\[
\varphi = \Re (\Omega) - dr \wedge \omega  \quad \text{on } S^1 \times X^6.
\]
In either case, the $G_2$-Laplacian flow exists for all time and converges as $t \rightarrow \infty$ to a stationary point $\varphi_\infty$ of the form
\begin{equation}
\label{eqn-varphi-flow-stationary}
    \varphi_{\infty} = 
    \begin{cases}
        - dr^1 \wedge dr^2 \wedge dr^3 + dr^1 \wedge \Theta_\infty^* \omega_{{\CY}} \\
        \qquad + dr^2 \wedge \Re (\Theta_\infty^*\Omega) + dr^3 \wedge \Im (\Theta_\infty^*\Omega) &\text{ on } T^3 \times X^4 \\[5pt]
        \Re( \Theta_\infty^*\Omega) - dr \wedge \Theta_\infty^* \omega_{{\CY}} &\text{ on } S^1 \times X^6,
    \end{cases}
\end{equation}
where $\Theta_\infty: X^{2n} \rightarrow X^{2n}$ is a diffeomorphism and $\omega_{\CY}$ is the unique K\"ahler Ricci-flat metric in the K\"{a}hler class $[\omega]$.
\end{thm}

A similar analysis can be done when starting with the $G_2$-Laplacian coflow
\begin{equation}
\label{eqn-Laplacian-coflow}
    \frac{d\psi_t}{dt} = {\Delta_d}_t \psi_t
\end{equation}
on $M^7 = T^3 \times X^4$ with initial data
\begin{equation}
\label{eqn-varphi-coflow-dim-2}
\psi = - 2^{-\frac{4}{3}} \frac{1}{2} \omega^2 + 2^{-\frac{4}{3}} dr^2 \wedge dr^3 \wedge \omega + 2^{\frac{2}{3}} dr^3 \wedge dr^1 \wedge \Re (\Omega) + 2^{\frac{2}{3}} dr^1 \wedge dr^2 \wedge \Im (\Omega),
\end{equation}
or on $M^7 = S^1 \times X^6$ with initial data
\begin{equation}
\label{eqn-varphi-coflow-dim-3}
 \psi = - 2dr \wedge \Im (\Omega) - \frac{1}{4} \frac{1}{2} \omega^2.
\end{equation}
From $\psi$, we recover $\varphi$ by $\psi = \star \varphi$ on $M^7$. This ansatz for $\varphi$ is natural for the coflow as $\varphi$ is coclosed ($d \psi=0$) but not closed.

In this case, it can be shown that a solution to the $G_2$-Laplacian coflow on $M^7$ is generated by the K\"{a}hler--Ricci flow on $X^{2n}$. With this set up, we will show that a solution for the 3-form $\varphi_t$ is given by
\begin{equation}
\label{eqn-ansatz-varphi-coflow}
    \varphi_t = 
    \begin{cases}
        - |\Omega_t|_{\omega_t} dr^1 \wedge dr^2 \wedge dr^3 + dr^1 \wedge |\Omega_t|_{\omega_t} \omega_t \\
        \qquad + dr^2 \wedge \Re \Big( \frac{1}{|\Omega_t|_{\omega_t}} \Omega_t \Big) + dr^3 \wedge \Im \Big( \frac{1}{|\Omega_t|_{\omega_t}} \Omega_t \Big) &\text{ on } T^3 \times X^4, \\[5pt]
        \Re \Big( \frac{1}{|\Omega_t|_{\omega_t}} \Omega_t \Big) - dr \wedge |\Omega_t|_{\omega_t} \omega_t &\text{ on } S^1 \times X^6,
    \end{cases}
\end{equation}
with $(\omega_t,\Omega_t) = (\Theta_t^* \wtilde{\omega}_t, \Theta_t^* \Omega)$ where $\wtilde{\omega}_t$ solves the (rescaled) K\"{a}hler--Ricci flow
\begin{equation}
\label{eqn-KR-flow}
    \frac{d\wtilde{\omega_t}}{dt} = -2K\Ric(\wtilde{\omega}_t,J), \quad \wtilde{\omega}_0 = \omega,
\end{equation}
on $X^{2n}$ with $K = 2^{\frac{n}{3}}$ and $\Theta_t$ is again a family of diffeomorphisms determined by $\wtilde{\omega}_t$ with $\Theta_0 = \id_{X^{2n}}$.

\begin{rmk}
It is intriguing that the solution to the $G_2$-Laplacian flow is generated by the MA$^{\frac{1}{3}}$ flow on the complex manifold, while the solution to the coflow is generated by the more well-known K\"ahler--Ricci flow on the complex manifold. From the point of view of Monge--Amp\`ere equations, the MA$^{\frac{1}{3}}$ flow involves the determinant to the power of $\frac{1}{3}$, while the K\"ahler-Ricci flow involves the logarithm of the determinant. 
\end{rmk}

Known results \cite{Cao85} regarding the K\"{a}hler--Ricci flow allow us to conclude the following:

\begin{thm}
\label{thm-Laplacian-coflow-KR}
Let $X^{2n}$ be a compact K\"ahler Calabi--Yau $n$-fold where $n = 2$ or $3$ with K\"ahler metric $\omega$ and holomorphic volume form $\Omega$. Start the $G_2$-Laplacian coflow with initial data
\begin{equation}
\label{eqn-psi-coflow-dim-2}
\psi = - 2^{-\frac{4}{3}} \frac{1}{2} \omega^2 + 2^{-\frac{4}{3}} dr^2 \wedge dr^3 \wedge \omega + 2^{\frac{2}{3}} dr^3 \wedge dr^1 \wedge \Re (\Omega) + 2^{\frac{2}{3}} dr^1 \wedge dr^2 \wedge \Im (\Omega) \text{ on } T^3 \times X^4,
\end{equation}
or
\begin{equation}
\label{eqn-psi-coflow-dim-3}
 \psi = - 2 dr \wedge \Im (\Omega) - \frac{1}{4} \frac{1}{2} \omega^2 \text{ on } S^1 \times X^6
\end{equation}
Then a solution to the $G_2$-Laplacian coflow exists for all time and converges as $t \rightarrow \infty$ to a stationary point $\psi_\infty$ of the form
\begin{equation}
\label{eqn-varphi-coflow-stationary}
    \psi_{\infty} = 
    \begin{cases}
        - 2^{-\frac{4}{3}} \frac{1}{2} \Theta_\infty^* \omega^2_{\CY} + 2^{-\frac{4}{3}} dr^2 \wedge dr^3 \wedge \Theta_\infty^* \omega_{\CY} \\
        \qquad + 2^{\frac{2}{3}} dr^3 \wedge dr^1 \wedge \Re (\Theta_\infty^*\Omega) + 2^{\frac{2}{3}} dr^1 \wedge dr^2 \wedge \Im (\Theta_\infty^*\Omega)  &\text{ on } T^3 \times X^4, \\[5pt]
       - 2dr \wedge \Im (\Theta_\infty^*\Omega) - \frac{1}{4} \frac{1}{2} \Theta_\infty^*\omega_{\CY}^2  &\text{ on } S^1 \times X^6,
    \end{cases}
\end{equation}
where $\Theta_\infty: X^{2n} \rightarrow X^{2n}$ is a diffeomorphism and $\omega_{\CY}$ is the unique K\"ahler Ricci-flat metric in the K\"ahler class $[\omega]$.
\end{thm}

In summary, both the $G_2$-Laplacian flow and coflow on $T^3 \times X^4$ or $S^1 \times X^6$ take an arbitrary K\"ahler metric $\omega$ and create a path $\omega_t$ on $X$ converging to the unique K\"ahler Ricci-flat metric $\omega_{\CY} \in [\omega]$ by the ansatz given above. In this way, we see that $G_2$ flows can be used to give a $G_2$ proof of Yau's theorem \cite{Yau78} on Calabi--Yau $2$ and $3$-folds.

The Einstein summation convention will be employed throughout.

\par {\bf Acknowledgements:} We thank S. Karigiannis for helpful discussions. S.P. thanks T. Fei, D.H. Phong, and X.-W. Zhang for previous collaborations on geometric flows in special geometries. We also thank the referee for helpful comments and a careful reading.

\section{Preliminaries on \texorpdfstring{$G_2$}{G2}-Structures}
\label{sect-G2-prelims}

\subsection{Review of \texorpdfstring{$G_2$}{G2}-Structures}
\label{sect-G2-Structs}

In this section, we review some facts from $G_2$ geometry. For a more in depth reference, see \cite{Kar20}.

We start with a definition.

\begin{defn}
\label{defn-G2-Struct}
A $3$-form on a $7$-dimensional manifold $M$ is called positive if for any non-zero $Y \in T_pM$,
\begin{equation}
\label{eqn-pos-3-form}
    (Y \hook \varphi) \wedge (Y \hook \varphi) \wedge \varphi \neq 0.
\end{equation}
\end{defn}

A smooth positive $3$-form is also called a $G_2$-structure. In this case, $\varphi$ induces a unique metric $g$ and volume form $\vol$ by the relation:
\begin{equation}
\label{eqn-G2-metric-vol}
    -\frac{1}{6} (Y \hook \varphi) \wedge (Z \hook \varphi) \wedge \varphi = g(Y,Z) \vol.
\end{equation}

The volume form $\vol$ in \eqref{eqn-G2-metric-vol} is the Riemannian volume form of the induced metric $g$.

\begin{rmk}
\label{rmk-sign}
We remark that some authors choose a different sign convention (and orientation) by flipping the sign in \eqref{eqn-G2-metric-vol}.
\end{rmk}

With these induced structures, we get a Hodge star operator $\star$ and a $4$-form $\psi$ dual to $\varphi$ given by $\psi = \star \varphi$.

We note here that $7$-dimensional manifold admits a $G_2$-structure if and only if it is orientable and spinnable.

\begin{defn}
\label{defn-closed-coclosed}
A $G_2$-structure $\varphi$ is called:
\begin{itemize}
    \item closed if $d\varphi = 0$,
    \item coclosed if $d\psi = 0$,
    \item torsion-free if it is both closed and coclosed.
\end{itemize}
\end{defn}

Under the action of $G_2$, we can decompose the spaces of forms on $M$. More details on these decompositions can be found in \cite{Kar09}. In particular, we can apply them to the $4$-form $d\varphi$ and $5$-form $d\psi$, which in turn allows us to define the torsion forms.

\begin{defn}
\label{defn-torsion}
The torsion forms of a $G_2$-structure $\varphi$ are 
\begin{equation*}
    \tau_0 \in \Omega^0_1, \quad \tau_1 \in \Omega^1_7, \quad \tau_2 \in \Omega^2_{14}, \quad \tau_3 \in \Omega^3_{27},
\end{equation*}
and are defined by the equations
\begin{equation}
\label{eqn-torsion}
\begin{aligned}
    d\varphi &= \tau_0 \psi + 3 \tau_1 \wedge \varphi + \star \tau_3, \\
    d\psi &= 4 \tau_1 \wedge \psi + \star \tau_2.
\end{aligned}
\end{equation}
\end{defn}

The torsion forms can be computed using the following identities:
\begin{equation}
\label{eqn-torsion-id}
\begin{aligned}
    \tau_0 &= \frac{1}{7} \star (\varphi \wedge d\varphi), \\
    \tau_1 &= \frac{1}{12} \star (\varphi \wedge \star d\varphi) = \frac{1}{12} \star (\psi \wedge \star d\psi).
\end{aligned}
\end{equation}

It can be shown that the $\tau_1$ that appears in both equations \eqref{eqn-torsion} are indeed the same. A proof of this can be found in \cite{Kar09}.

\subsection{Calabi--Yau \texorpdfstring{$2$}{2}-Folds and their Associated \texorpdfstring{$G_2$}{G2}-Structures}
\label{subsect-CY-2-Folds}

We now briefly describe a manner in which we can obtain a $G_2$-structure from a smooth Calabi--Yau $2$-fold $X^4$ by crossing it with the $3$-torus $T^3$. Let $X^4$ be a compact smooth Calabi--Yau $2$-fold. Let $\omega$ be a K\"{a}hler form and $\Omega$ be a nowhere vanishing holomorphic $(2,0)$-form. Both $\omega$ and $\Omega$ are closed. In local coordinates we may write
\begin{equation}
\label{eqn-omega-local-dim-2}
    \omega = i {(g_4)}_{p\overline{q}} dz^p \wedge d\overline{z}^q,
\end{equation}
and
\begin{equation}
\label{eqn-Omega-local-dim-2}
    \Omega = f dz^1 \wedge dz^2
\end{equation}
where $g_4 = {(g_4)}_{p\overline{q}}$ is the metric associated to $\omega$ and $f$ is a local holomorphic function. We have $\omega(\cdot,\cdot)=g_4(J \cdot, \cdot)$. The norm of $\Omega$ with respect to $\omega$ can be computed by the formula
\begin{equation}
\label{eqn-norm-Omega-dim-2}
    |\Omega|_\omega^2 = \frac{|f|^2}{\det {(g_4)}_{p\overline{q}}}.
\end{equation}
If the K\"{a}hler metric $\omega$ is Ricci-flat \cite{Yau78}, this norm is a constant, however we need this formula for an arbitrary K\"{a}hler metric $\omega$.

The pair $(\omega, \Omega)$ satisfies the following relations:
\begin{equation}
    \frac{\omega^2}{2!} = \vol_4 = \frac{1}{|\Omega|_\omega^2} \Omega \wedge \overline{\Omega} = 2 \Re \Big( \frac{1}{|\Omega|_\omega} \Omega \Big) \wedge \Re \Big( \frac{1}{|\Omega|_\omega} \Omega \Big) = 2\Im \Big( \frac{1}{|\Omega|_\omega} \Omega \Big) \wedge \Im \Big( \frac{1}{|\Omega|_\omega} \Omega \Big),
\end{equation}
where $\vol_4$ is the volume form on $(X^4,g_4)$. In addition, we have these identities involving the Hodge star operator $\star_4$ on $X^4$:
\begin{equation}
\label{eqn-star_4}
    (\star_4)^2 \alpha = (-1)^k \alpha \text{ for } \alpha \in \Omega^k(X^4), 
    \quad \star_4 \Re (\Omega) = \Re (\Omega), \quad \star_4 \Im (\Omega) = \Im (\Omega), \quad \star_4 \omega = \omega.
\end{equation}

Let $r^1, r^2$, and $r^3$ denote the angle coordinates on $T^3$. If $F$ is a smooth nowhere-vanishing complex function on $X^4$ and $G$ is a smooth strictly positive function on $X^4$, we can consider the $3$-form $\varphi$ on $M^7 = T^3 \times X^4$ given by
\begin{equation}
\label{eqn-varphi-dim-2}
    \varphi = -G dr^1 \wedge dr^2 \wedge dr^3 + dr^1 \wedge G \omega + dr^2 \wedge \Re \Big( \frac{F}{|\Omega|_\omega} \Omega \Big) + dr^3 \wedge \Im \Big( \frac{F}{|\Omega|_\omega} \Omega \Big).
\end{equation}

Routine computations show that $\varphi$ is a positive $3$-form which yields a $G_2$-structure on $M^7$. It can be shown that in this case, the metric $g_7$ on $M^7$ induced by $\varphi$ is
\begin{equation}
\label{eqn-g_7-dim-2}
    g_7 = 2^{\frac{4}{3}} |F|^{-\frac{4}{3}} G^2 (dr^1)^2 + 2^{-\frac{2}{3}}|F|^\frac{2}{3} (dr^2)^2 + 2^{-\frac{2}{3}}|F|^\frac{2}{3} (dr^3)^3 + 2^{-\frac{2}{3}}|F|^\frac{2}{3} g_4,
\end{equation}
and the volume form $\vol_7$ is
\begin{equation}
\label{eqn-vol_7-dim-2}
    \vol_7 = 2^{-\frac{4}{3}}|F|^{\frac{4}{3}} G dr^1 \wedge dr^2 \wedge dr^3 \wedge \vol_4.
\end{equation}

Using these formulae, we can compute the following Hodge star identities on $M^7$: For $\alpha \in \Omega^k(X^4)$, we have
\begin{equation}
\label{eqn-star_7-dim-2}
\begin{aligned}
    \star_7 \alpha &= (-1)^k 2^{(-\frac{4}{3} + \frac{2}{3}k)} |F|^{(\frac{4}{3} - \frac{2}{3}k)} G dr^1 \wedge dr^2 \wedge dr^3 \wedge \star_4 \alpha, \\
    \star_7 (dr^1 \wedge \alpha) &= 2^{(-\frac{8}{3} + \frac{2}{3}k)} |F|^{(\frac{8}{3} - \frac{2}{3}k)} G^{-1} dr^2 \wedge dr^3 \wedge \star_4 \alpha, \\
    \star_7 (dr^2 \wedge \alpha) &= 2^{(-\frac{8}{3} + \frac{2}{3}k)}  |F|^{(\frac{2}{3} - \frac{2}{3}k)} G dr^3 \wedge dr^1 \wedge \star_4 \alpha, \\
    \star_7 (dr^3 \wedge \alpha) &= 2^{(-\frac{2}{3} + \frac{2}{3}k)} |F|^{(\frac{2}{3} - \frac{2}{3}k)} G dr^1 \wedge dr^2 \wedge \star_4 \alpha, \\
    \star_7 (dr^1 \wedge dr^2 \wedge \alpha) &= (-1)^k 2^{(-2 + \frac{2}{3}k)} |F|^{(2 - \frac{2}{3}k)} G^{-1} dr^3 \wedge \star_4 \alpha, \\
    \star_7 (dr^3 \wedge dr^1 \wedge \alpha) &= (-1)^k 2^{(-2 + \frac{2}{3}k)} |F|^{(2 - \frac{2}{3}k)} G^{-1} dr^2 \wedge \star_4 \alpha, \\
    \star_7 (dr^2 \wedge dr^3 \wedge \alpha) &= (-1)^k 2^{\frac{2}{3}k}  |F|^{- \frac{2}{3}k} G dr^1 \wedge \star_4 \alpha, \\
    \star_7 (dr^1 \wedge dr^2 \wedge dr^3 \wedge \alpha) &= 2^{(-\frac{4}{3} + \frac{2}{3}k)} 
 |F|^{(\frac{4}{3} - \frac{2}{3}k)} G^{-1} \star_4 \alpha.
\end{aligned}
\end{equation}
Using the above and \eqref{eqn-star_4}, we get that the dual $4$-form in this case is given by
\begin{equation}
\label{eqn-psi-dim-2}
\begin{aligned}
    \psi &= -2^{-\frac{4}{3}} |F|^\frac{4}{3} \vol_4 + 2^{-\frac{4}{3}} |F|^\frac{4}{3} dr^2 \wedge dr^3 \wedge \omega \\
    &\qquad + 2^{\frac{2}{3}} |F|^{-\frac{2}{3}} G dr^3 \wedge dr^1 \wedge \Re \Big( \frac{F}{|\Omega|_\omega} \Omega \Big) + 2^{\frac{2}{3}} |F|^{-\frac{2}{3}} G dr^1 \wedge dr^2 \wedge \Im \Big( \frac{F}{|\Omega|_\omega} \Omega \Big).
\end{aligned}
\end{equation}

\subsection{Calabi--Yau \texorpdfstring{$3$}{3}-Folds and their Associated \texorpdfstring{$G_2$}{G2}-Structures}
\label{subsect-CY-3-Folds}

In a similar manner to the previous section, we can obtain a $G_2$-structure from a smooth Calabi--Yau $3$-fold $X^6$ by crossing it with the circle $S^1$. Let $X^6$ be a compact smooth Calabi--Yau $3$-fold, and let $\omega$ be a K\"{a}hler form and $\Omega$ be a nowhere vanishing holomorphic $(3,0)$-form. As in the $2$-fold case, both $\omega$ and $\Omega$ are closed, and we have the following analogous local descriptions of these forms in holomorphic coordinates:
\begin{equation}
\label{eqn-omega-local-dim-3}
    \omega = i {(g_6)}_{p\bar{q}} dz^p \wedge d\bar{z}^q,
\end{equation}
\begin{equation}
\label{eqn-Omega-local-dim-3}
    \Omega = f dz^1 \wedge dz^2 \wedge dz^3,
\end{equation}
where $g_6 = {(g_6)}_{p \bar{q}}$ is the metric associated to $\omega$ and $f$ is a local holomorphic function. The norm of $\Omega$ with respect to $\omega$ is given by the formula
\begin{equation}
\label{eqn-norm-Omega-dim-3}
    |\Omega|_\omega^2 = \frac{|f|^2}{\det {(g_6)}_{p \bar{q}}}
\end{equation}
This norm is again a constant when $\omega$ is K\"ahler Ricci-flat. 

The pair $(\omega,\Omega)$ satisfies the following relations:
\begin{equation*}
    \frac{\omega^3}{3!} = \vol_6 = i \frac{1}{|\Omega|_\omega^2} \Omega \wedge \bar{\Omega} = 2 \Re \Big(\frac{1}{|\Omega|_\omega} \Omega \Big) \wedge \Im \Big(\frac{1}{|\Omega|_\omega} \Omega \Big),
\end{equation*}
where $\vol_6$ is the volume form on $(X^6,g_6)$.

The following identities all hold, where $\star_6$ denotes the Hodge star operator on $X^6$:
\begin{equation}
\label{eqn-star_6}
\begin{aligned}
    (\star_6)^2 \beta = (-1)^k \beta \text{ for } \beta \in \Omega^k(X^6),\quad \star_6 \Re(\Omega) = \Im(\Omega), \quad \star_6 \omega = \frac{1}{2}\omega^2.
\end{aligned}
\end{equation}

Let $r$ denote the angle coordinate on $S^1$ so $dr$ is the globally defined volume form on $S^1$ with respect to the standard round metric. If $F$ is a smooth nowhere-vanishing complex function on $X^6$ and $G$ is a smooth strictly positive function on $X^6$, we can consider the $3$-form $\varphi$ on $M^7 = S^1 \times X^6$ given by
\begin{equation}
\label{eqn-varphi-dim-3}
    \varphi = \Re \Big(\frac{F}{|\Omega|_\omega} \Omega \Big) - dr \wedge G \omega.
\end{equation}

As seen in \cite{KMT12}, $\varphi$ is a positive $3$-form and thus is a $G_2$-structure on $M^7$. The $3$-form $\varphi$ induces the metric
\begin{equation}
\label{eqn-g_7-dim-3}
    g_7 = 4|F|^{-\frac{4}{3}} G^2 dr^2 + \frac{1}{2}|F|^{\frac{2}{3}} g_6.
\end{equation}
It is also seen that the associated volume form is
\begin{equation}
\label{eqn-vol_7-dim-3}
    \vol_7 = \frac{1}{4}|F|^{\frac{4}{3}} G dr \wedge \vol_6.
\end{equation}

With the expressions \eqref{eqn-g_7-dim-3} and \eqref{eqn-vol_7-dim-3} for the metric and volume form, one can compute that if $\beta \in \Omega^k(X^6)$, then the Hodge star $\star_7$ on $M^7$ is given by
\begin{equation}
\label{eqn-star_7-dim-3}
\begin{aligned}
    \star_7 \beta &= (-1)^k 2^{(-2+k)}|F|^{(\frac{4}{3} - \frac{2}{3}k)} G dr \wedge \star_6 \beta, \\
    \star_7 (dr \wedge \beta) &= 2^{(-4+k)}|F|^{(\frac{8}{3} - \frac{2}{3}k)} G^{-1} \star_6 \beta.
\end{aligned}
\end{equation}

From this and \eqref{eqn-star_6}, we see that the dual $4$-form $\psi$ is given by
\begin{equation}
\label{eqn-psi-dim-3}
    \psi = -2|F|^{-\frac{2}{3}} G dr \wedge \Im \Big(\frac{F}{|\Omega|_\omega} \Omega \Big) - \frac{1}{4} |F|^\frac{4}{3} \frac{1}{2} \omega^2.
\end{equation}  

\section{Laplacian Flow of \texorpdfstring{$G_2$}{G2}-Structures}
\label{sect-Laplacian-Flow}

We want to study the Laplacian flow of the $G_2$-structure obtained by a specific choice of functions $F$ and $G$. This flow was introduced by Bryant \cite{Bry05} and has since generated much activity, see \cite{BV20,FR20,FY18,LL21,Lot20,LW17,LW19} and references therein.

\begin{defn}
\label{defn-Laplacian-flow}
A time-dependent $G_2$-structure $\varphi_t$ on a $7$-dimensional manifold $M$ defined on some interval $[0,T)$ satisfies the Laplacian flow equation if
\begin{equation}
\label{eqn-Laplacian-flow-2}
    \frac{d \varphi_t}{d t} = {\Delta_d}_t \varphi_t
\end{equation}
wherever $\varphi_t$ is defined. Here ${\Delta_d}_t = dd_t^* + d_t^*d$ is the Hodge Laplacian with respect to the metric $(g_7)_t$ induced from $\varphi_t$.
\end{defn}

The $G_2$-Laplacian flow has been shown to have short-time existence and uniqueness when starting with an initially closed $G_2$-structure \cite{BX04} (see also \cite{BV20,Gri13,Lot20}). Such a solution preserves the closedness condition. We see that the stationary points of the Laplacian flow are torsion-free structures when $M$ is compact. Though we are primarily focused on the case where $M$ is compact, we note that torsion-free $G_2$ structures are also stationary points of the Laplacian flow even in the non-compact case \cite{Lot20}.

\subsection{The Hodge Laplacian of \texorpdfstring{$\varphi$}{varphi} on \texorpdfstring{$T^3 \times X^4$}{T3 x X4}}
\label{subsect-Hodge-Laplacian-varphi-dim-2}

Consider our first setting where $M^7 = T^3 \times X^4$ described in \S \ref{subsect-CY-2-Folds}. We can choose the functions $F$ and $G$ such that the $3$-form $\varphi$ is closed. In particular, we set $F = |\Omega|_\omega$ and $G = 1$ to get
\begin{equation}
\label{eqn-varphi-dim-2-2}
    \varphi = -dr^1 \wedge dr^2 \wedge dr^3 + dr^1 \wedge \omega + dr^2 \wedge \Re (\Omega) + dr^3 \wedge \Im (\Omega).
\end{equation}

Using our formulae from \S \ref{subsect-CY-2-Folds}, we have that the metric $g_7$ on $M^7$ is
\begin{equation}
\label{eqn-g_7-dim-2-2}
    g_7 = 2^{\frac{4}{3}}|\Omega|_\omega^{-\frac{4}{3}} (dr^1)^2 + 2^{-\frac{2}{3}} |\Omega|_\omega^\frac{2}{3} (dr^2)^2 + 2^{-\frac{2}{3}} |\Omega|_\omega^\frac{2}{3} (dr^3)^3 + 2^{-\frac{2}{3}} |\Omega|_\omega^\frac{2}{3} g_4,
\end{equation}
and the volume form $\vol_7$ is
\begin{equation}
\label{eqn-vol_7-dim-2-2}
    \vol_7 = 2^{-\frac{4}{3}} |\Omega|_\omega^{\frac{4}{3}} dr^1 \wedge dr^2 \wedge dr^3 \wedge \vol_4.
\end{equation}

From \eqref{eqn-star_7-dim-2}, we have the following identities for the Hodge star $\star_7$ on $M^7$ where $\alpha \in \Omega^k(X^4)$ is a $k$-form:
\begin{equation}
\label{eqn-star_7-dim-2-2}
\begin{aligned}
    \star_7 \alpha &= (-1)^k 2^{(-\frac{4}{3} + \frac{2}{3}k)} |\Omega|_\omega^{(\frac{4}{3} - \frac{2}{3}k)} dr^1 \wedge dr^2 \wedge dr^3 \wedge \star_4 \alpha, \\
    \star_7 (dr^1 \wedge \alpha) &= 2^{(-\frac{8}{3} + \frac{2}{3}k)} |\Omega|_\omega^{(\frac{8}{3} - \frac{2}{3}k)} dr^2 \wedge dr^3 \wedge \star_4 \alpha, \\
    \star_7 (dr^2 \wedge \alpha) &= 2^{(-\frac{2}{3} + \frac{2}{3}k)} |\Omega|_\omega^{(\frac{2}{3} - \frac{2}{3}k)} dr^3 \wedge dr^1 \wedge \star_4 \alpha, \\
    \star_7 (dr^3 \wedge \alpha) &= 2^{(-\frac{2}{3} + \frac{2}{3}k)} |\Omega|_\omega^{(\frac{2}{3} - \frac{2}{3}k)} dr^1 \wedge dr^2 \wedge \star_4 \alpha, \\
    \star_7 (dr^1 \wedge dr^2 \wedge \alpha) &= (-1)^k 2^{(-2 + \frac{2}{3}k)} |\Omega|_\omega^{(2 - \frac{2}{3}k)} dr^3 \wedge \star_4 \alpha, \\
    \star_7 (dr^3 \wedge dr^1 \wedge \alpha) &= (-1)^k 2^{(-2 + \frac{2}{3}k)} |\Omega|_\omega^{(2 - \frac{2}{3}k)} dr^2 \wedge \star_4 \alpha, \\
    \star_7 (dr^2 \wedge dr^3 \wedge \alpha) &= (-1)^k 2^{\frac{2}{3}k} |\Omega|_\omega^{- \frac{2}{3}k} dr^1 \wedge \star_4 \alpha, \\
    \star_7 (dr^1 \wedge dr^2 \wedge dr^3 \wedge \alpha) &= 2^{(-\frac{4}{3} + \frac{2}{3}k)} |\Omega|_\omega^{(\frac{4}{3} - \frac{2}{3}k)} \star_4 \alpha.
\end{aligned}
\end{equation}
This allows us to see that the dual $4$-form $\psi$ in this case is given by
\begin{equation}
\label{eqn-psi-dim-2-2}
\begin{aligned}
    \psi &= -2^{-\frac{4}{3}} |\Omega|_\omega^\frac{4}{3} \vol_4 + 2^{-\frac{4}{3}} |\Omega|_\omega^\frac{4}{3} dr^2 \wedge dr^3 \wedge \omega \\
    &\qquad + 2^{\frac{2}{3}} |\Omega|_\omega^{-\frac{2}{3}} dr^3 \wedge dr^1 \wedge \Re (\Omega) + 2^{\frac{2}{3}} |\Omega|_\omega^{-\frac{2}{3}} dr^1 \wedge dr^2 \wedge \Im (\Omega).
\end{aligned} 
\end{equation}

Both $\omega$ and $\Omega$ are closed, and so $\varphi$ is a closed $G_2$-structure. We use this as our initial data for the $G_2$-Laplacian flow. In order to find the associated evolution equations, we compute the Hodge Laplacian of $\varphi$.

\begin{lem}
\label{lem-Hodge-Laplacian-varphi-flow-dim-2}
If $\varphi$ is the $G_2$-structure defined by \eqref{eqn-varphi-dim-2-2}, then
\begin{equation}
\label{eqn-Hodge-Laplacian-varphi-flow-dim-2}
    \Delta_d \varphi = 2^{\frac{2}{3}} \cal{L}_{\nabla (|\Omega|_\omega^{-\frac{2}{3}})} \Big(2 dr^1 \wedge \omega - dr^2 \wedge \Re (\Omega) - dr^3 \wedge \Im (\Omega) \Big).
\end{equation}
\end{lem}

\begin{proof}
Since $\varphi$ is closed, we have that
\begin{equation*}
    \Delta_d \varphi = dd^* \varphi + dd^* \varphi = dd^* \varphi.
\end{equation*}
Further, as $\varphi$ is a $3$-form on a $7$-dimensional manifold, we get that
\begin{equation*}
    dd^* \varphi = -d\star_7 d\star_7 \varphi = -d\star_7 d\psi.
\end{equation*}

We then compute
\begin{equation*}
\begin{aligned}
    d\psi &= d\Big( -2^{-\frac{4}{3}} |\Omega|_\omega^\frac{4}{3} \vol_4 + 2^{-\frac{4}{3}} |\Omega|_\omega^\frac{4}{3} dr^2 \wedge dr^3 \wedge \omega \\
    &\qquad \qquad + 2^{\frac{2}{3}} |\Omega|_\omega^{-\frac{2}{3}} dr^3 \wedge dr^1 \wedge \Re (\Omega) + 2^{\frac{2}{3}} |\Omega|_\omega^{-\frac{2}{3}} dr^1 \wedge dr^2 \wedge \Im (\Omega) \Big) \\
    &=  2^{\frac{2}{3}} \Big[ \frac{1}{3} |\Omega|_\omega^\frac{4}{3} d(\log |\Omega|_\omega) \wedge dr^2 \wedge dr^3 \wedge \omega \\
    &\qquad - \frac{2}{3} |\Omega|_\omega^{-\frac{2}{3}} d(\log |\Omega|_\omega) \wedge dr^3 \wedge dr^1 \wedge \Re (\Omega) - \frac{2}{3} |\Omega|_\omega^{-\frac{2}{3}} d(\log |\Omega|_\omega) \wedge dr^1 \wedge dr^2 \wedge \Im (\Omega) \Big].
\end{aligned}
\end{equation*}

Taking the Hodge star, and applying the identities
\begin{equation}
\label{eqn-hook-star}
    Y \hook \star \alpha = (-1)^k \star (Y^\flat \wedge \alpha) \text{ for } \alpha \in \Omega^k(M) \text{ and } Y \in \Gamma(TM),
\end{equation}
and
\begin{equation}
\label{eqn-grad-flat}
    df = (\nabla f)^\flat \text{ for } f \in \cal{C}^\infty(M),
\end{equation}
where $M$ is an arbitrary smooth manifold, we get that
\begin{equation*}
\begin{aligned}
    \star_7 d\psi &=  2^{\frac{2}{3}} \Big( \frac{1}{3} |\Omega|_\omega^\frac{4}{3} \Big(d(\log |\Omega|_\omega)\Big)^{\sharp_7} \hook \star_7 \Big[ dr^2 \wedge dr^3 \wedge \omega \Big] \\
    &\qquad -\frac{2}{3} |\Omega|_\omega^{-\frac{2}{3}} \Big(d(\log |\Omega|_\omega)\Big)^{\sharp_7} \hook \star_7 \Big[ dr^3 \wedge dr^1 \wedge \Re (\Omega) \Big] - \frac{2}{3} |\Omega|_\omega^{-\frac{2}{3}} \Big(d(\log |\Omega|_\omega)\Big)^{\sharp_7} \hook \star_7 \Big[ dr^1 \wedge dr^2 \wedge \Im (\Omega) \Big] \Big) \\
    &= 2^{\frac{2}{3}} \Big( \frac{1}{3} |\Omega|_\omega^\frac{4}{3} \Big( 2^{\frac{2}{3}} |\Omega|_\omega^{-\frac{2}{3}} \nabla_{(g_4)} (\log |\Omega|_\omega) \Big) \hook \Big[ 2^{\frac{4}{3}} |\Omega|_\omega^{-\frac{4}{3}} dr^1 \wedge \omega \Big] \\
    &\qquad -\frac{2}{3} |\Omega|_\omega^{-\frac{2}{3}} \Big( 2^{\frac{2}{3}} |\Omega|_\omega^{-\frac{2}{3}} \nabla_{(g_4)} (\log |\Omega|_\omega) \Big) \hook \Big[ 2^{-\frac{2}{3}} |\Omega|_\omega^\frac{2}{3} dr^2 \wedge \Re (\Omega) \Big] \\
    &\qquad - \frac{2}{3} |\Omega|_\omega^{-\frac{2}{3}} \Big( 2^{\frac{2}{3}} |\Omega|_\omega^{-\frac{2}{3}} \nabla_{(g_4)} (\log |\Omega|_\omega) \Big) \hook \Big[ 2^{-\frac{2}{3}} |\Omega|_\omega^\frac{2}{3} dr^3 \wedge \Im (\Omega) \Big] \Big) \\
    &= 2^{\frac{2}{3}} \Big( \frac{4}{3} |\Omega|_\omega^{-\frac{2}{3}} \Big( \nabla_{(g_4)} (\log |\Omega|_\omega) \Big) \hook \Big[ dr^1 \wedge \omega \Big] \\
    &\qquad -\frac{2}{3} |\Omega|_\omega^{-\frac{2}{3}} \Big( \nabla_{(g_4)} (\log |\Omega|_\omega) \Big) \hook \Big[ dr^2 \wedge \Re (\Omega) \Big] - \frac{2}{3} |\Omega|_\omega^{-\frac{2}{3}} \Big( \nabla_{(g_4)} (\log |\Omega|_\omega) \Big) \hook \Big[ dr^3 \wedge \Im (\Omega) \Big] \Big) \\
    &= - 2^{\frac{2}{3}} \Big( \frac{2}{3} |\Omega|_\omega^{-\frac{2}{3}} \Big( \nabla_{(g_4)} (\log |\Omega|_\omega) \Big) \hook \Big[ -2dr^1 \wedge \omega + dr^2 \wedge \Re (\Omega) + dr^3 \wedge \Im (\Omega) \Big] \Big) \\
    &= 2^{\frac{2}{3}} \Big( \nabla_{(g_4)} (|\Omega|_\omega^{-\frac{2}{3}}) \Big) \hook \Big[ -2 dr^1 \wedge \omega + dr^2 \wedge \Re (\Omega) + dr^3 \wedge \Im (\Omega) \Big].
\end{aligned}
\end{equation*}

We can use Cartan's magic formula
\begin{equation}
\label{eqn-Cartan's-magic-formula}
    \cal{L}_Y \alpha = d(Y \hook \alpha) + Y \hook (d\alpha) \text{ for } \alpha \in \Omega^k(M) \text{ and } Y \in \Gamma(TM),
\end{equation}
and the closedness of the forms $\omega$ and $\Omega$ to conclude that
\begin{equation*}
    \Delta_d \varphi = -d\star_7 d\psi = 2^{\frac{2}{3}} \cal{L}_{\nabla_{(g_4)} (|\Omega|_\omega^{-\frac{2}{3}})} \Big( 2dr^1 \wedge \omega - dr^2 \wedge \Re (\Omega) - dr^3 \wedge \Im (\Omega) \Big)
\end{equation*}
as desired.
\end{proof}

We can use the intermediate expressions in the proof of Lemma \ref{lem-Hodge-Laplacian-varphi-flow-dim-2} to compute the torsion forms of $\varphi$.

\begin{lem}
\label{lem-torsion-flow-dim-2}
If $\varphi$ is the $G_2$-structure defined by \eqref{eqn-varphi-dim-2-2}, then the torsion forms are given by
\begin{equation}
\label{eqn-torsion-flow-dim-2}
    \tau_0 = 0,\quad \tau_1 = 0,\quad \tau_2 = 2^{\frac{2}{3}} \Big( \nabla_{(g_4)} (|\Omega|_\omega^{-\frac{2}{3}}) \Big) \hook \Big[ -2dr^1 \wedge \omega + dr^2 \wedge \Re (\Omega) + dr^3 \wedge \Im (\Omega) \Big] ,\quad \tau_3 = 0.
\end{equation}
\end{lem}

\begin{proof}
Since $\varphi$ is closed, we have $d\varphi = 0$. It follows from \eqref{eqn-torsion} that each of $\tau_0, \tau_1,$ and $\tau_3$ vanish. Next, using that $\tau_1 = 0$ and \eqref{eqn-torsion} we see that
\begin{equation*}
\begin{aligned}
    \tau_2 = (\star_7)^2 \tau_2 = \star_7 d\psi = 2^{\frac{2}{3}} \Big( \nabla_{(g_4)} (|\Omega|_\omega^{-\frac{2}{3}}) \Big) \hook \Big[ -2dr^1 \wedge \omega + dr^2 \wedge \Re (\Omega) + dr^3 \wedge \Im (\Omega) \Big].
\end{aligned}
\end{equation*}
\end{proof}

\subsection{The Hodge Laplacian of \texorpdfstring{$\varphi$}{varphi} on \texorpdfstring{$S^1 \times X^6$}{S1 x X6}}
\label{subsect-Hodge-Laplacian-varphi-dim-3}

We return to our second framework where $M^7 = S^1 \times X^6$ as described in \S \ref{subsect-CY-3-Folds}. Recall that the K\"{a}hler form $\omega$ is a closed real $(1,1)$-form on $X^6$ and that $\Omega$ is a closed nowhere-vanishing holomorphic $(3,0)$-form. Using \eqref{eqn-varphi-dim-3}, we can again set $F = |\Omega|_\omega$ and $G = 1$ to define a $G_2$-structure on $M^7$ given by
\begin{equation}
\label{eqn-varphi-dim-3-2}
    \varphi = \Re (\Omega) - dr \wedge \omega.
\end{equation}

In this case, the formulas from the \S \ref{subsect-CY-3-Folds} give that the metric $g_7$ on $M^7$ is
\begin{equation}
\label{eqn-g_7-dim-3-2}
    g_7 = 4|\Omega|_\omega^{-\frac{4}{3}} dr^2 + \frac{1}{2}|\Omega|_\omega^\frac{2}{3} g_6,
\end{equation}
the volume form $\vol_7$ is
\begin{equation}
\label{eqn-vol_7-dim-3-2}
    \vol_7 = \frac{1}{4}|\Omega|_\omega^\frac{4}{3} dr \wedge \vol_6,
\end{equation}
and the dual $4$-form $\psi$ is
\begin{equation}
\label{eqn-psi-dim-3-2}
    \psi = -2|\Omega|_\omega^{-\frac{2}{3}} dr \wedge \Im(\Omega) - \frac{1}{4} |\Omega|_\omega^\frac{4}{3} \frac{1}{2} \omega^2.
\end{equation}

We also see from \eqref{eqn-star_7-dim-3} that the Hodge star $\star_7$ acts on forms by
\begin{equation}
\label{eqn-star_7-dim-3-2}
\begin{aligned}
    \star_7 \beta &= (-1)^k 2^{(-2+k)} |\Omega|_\omega^{(\frac{4}{3} - \frac{2}{3}k)} dr \wedge \star_6 \beta, \\
    \star_7 (dr \wedge \beta) &= 2^{(-4+k)} |\Omega|_\omega^{(\frac{8}{3} - \frac{2}{3}k)} \star_6 \beta,
\end{aligned}
\end{equation}
where $\beta \in \Omega^k(X^6)$ is a $k$-form on $X^6$.

Since both $\omega$ and $\Omega$ are closed, we have that $\varphi$ defines a closed $G_2$-structure. We now compute the Hodge Laplacian of $\varphi$.

\begin{lem}
\label{lem-Hodge-Laplacian-varphi-flow-dim-3}
If $\varphi$ is the $G_2$-structure defined by \eqref{eqn-varphi-dim-3-2}, then
\begin{equation}
\label{eqn-Hodge-Laplacian-varphi-flow}
    \Delta_d \varphi = 2 \cal{L}_{\nabla (|\Omega|_\omega^{-\frac{2}{3}})} \Big( -\Re(\Omega) - 2 dr \wedge \omega\Big).
\end{equation}
\end{lem}

\begin{proof}
We have that $\varphi$ is closed, and so as noted in the proof of Lemma \ref{lem-Hodge-Laplacian-varphi-flow-dim-2},
\begin{equation*}
    \Delta_d \varphi = -d\star_7 d\psi.
\end{equation*}

We compute
\begin{equation*}
\begin{aligned}
    d\psi &= d\Big(-2 |\Omega|_\omega^{-\frac{2}{3}} dr \wedge \Im(\Omega) - \frac{1}{4} |\Omega|_\omega^\frac{4}{3} \frac{1}{2} \omega^2\Big) \\
    &= \frac{4}{3} |\Omega|_\omega^{-\frac{2}{3}} d(\log |\Omega|_\omega) \wedge dr \wedge \Im(\Omega) - \frac{1}{3} |\Omega|_\omega^\frac{4}{3} d(\log |\Omega|_\omega) \wedge \frac{1}{2} \omega^2.
\end{aligned}
\end{equation*}

Taking the Hodge star and applying the identities \eqref{eqn-hook-star} and \eqref{eqn-grad-flat}, we get
\begin{equation*}
\begin{aligned}
    \star_7 d\psi &= \frac{4}{3} |\Omega|_\omega^{-\frac{2}{3}} \star_7 \Big[ d(\log |\Omega|_\omega) \wedge dr \wedge \Im(\Omega)\Big] - \frac{1}{3} |\Omega|_\omega^\frac{4}{3} \star_7 \Big[ d(\log |\Omega|_\omega) \wedge \frac{1}{2} \omega^2\Big] \\
    &= \frac{4}{3} |\Omega|_\omega^{-\frac{2}{3}} \Big(d (\log |\Omega|_\omega) \Big)^{\sharp_7} \hook \star_7 \Big[ dr \wedge \Im(\Omega)\Big] - \frac{1}{3} |\Omega|_\omega^\frac{4}{3} \Big(d (\log |\Omega|_\omega) \Big)^{\sharp_7} \hook \star_7 \Big[\frac{1}{2} \omega^2\Big] \\
    &= \frac{4}{3} |\Omega|_\omega^{-\frac{2}{3}} \Big( 2 |\Omega|_\omega^{-\frac{2}{3}} \nabla_{(g_6)} (\log |\Omega|_\omega) \Big) \hook \Big[ - \frac{1}{2} |\Omega|_\omega^\frac{2}{3} \Re(\Omega)\Big] - \frac{1}{3} |\Omega|_\omega^\frac{4}{3} \Big( 2 |\Omega|_\omega^{-\frac{2}{3}} \nabla_{(g_6)} (\log |\Omega|_\omega) \Big) \hook \Big[ 4 |\Omega|_\omega^{-\frac{4}{3}} dr \wedge \omega\Big] \\
    &= - \frac{4}{3} |\Omega|_\omega^{-\frac{2}{3}} \Big(\nabla_{(g_6)} (\log |\Omega|_\omega) \Big) \hook \Big[ \Re(\Omega) + 2 dr \wedge \omega\Big] \\
    &= 2 \Big(\nabla_{(g_6)} ( |\Omega|_\omega^{-\frac{2}{3}}) \Big) \hook \Big[ \Re(\Omega) + 2 dr \wedge \omega \Big].
\end{aligned}
\end{equation*}

Using Cartan's magic formula \eqref{eqn-Cartan's-magic-formula} and the closedness of $\omega$ and $\Omega$, we get that
\begin{equation*}
    \Delta_d \varphi = -d \star_7 d\psi = 2 \cal{L}_{\nabla_{(g_6)} ( |\Omega|_\omega^{-\frac{2}{3}}) } \Big( -\Re(\Omega) - 2 dr \wedge \omega\Big)
\end{equation*}
as desired.
\end{proof}

We can again use the intermediate expressions to compute the torsion forms of $\varphi$. In particular, we have

\begin{lem}
\label{lem-torsion-flow-dim-3}
If $\varphi$ is the $G_2$-structure defined by \eqref{eqn-varphi-dim-3-2}, then the torsion forms are given by
\begin{equation}
\label{eqn-torsion-flow-dim-3}
    \tau_0 = 0,\quad \tau_1 = 0,\quad \tau_2 = 2 \Big(\nabla_{(g_6)} ( |\Omega|_\omega^{-\frac{2}{3}}) \Big) \hook \Big[ \Re(\Omega) + 2 dr \wedge \omega\Big],\quad \tau_3 = 0.
\end{equation}
\end{lem}

\subsection{The Evolution Equations}
\label{subsect-evolution-eqns-flow}

We now start the Laplacian flow. Assuming the ansatz \eqref{eqn-varphi-dim-2-2} on $T^3 \times X^4$ is preserved, by the result of Lemma \ref{lem-Hodge-Laplacian-varphi-flow-dim-2}, we have the evolution equation
\begin{equation}
\label{eqn-Laplacian-flow-dim-2-3}
\begin{aligned}
    &\frac{d}{dt} \Big( -dr^1 \wedge dr^2 \wedge dr^3 + dr^1 \wedge \omega + dr^2 \wedge \Re (\Omega) + dr^3 \wedge \Im (\Omega) \Big) \\
    &\qquad = 2^{\frac{2}{3}} \cal{L}_{\nabla_{(g_4)} (|\Omega|_\omega^{-\frac{2}{3}})} \Big( 2dr^1 \wedge \omega - dr^2 \wedge \Re (\Omega) - dr^3 \wedge \Im (\Omega) \Big).
\end{aligned}
\end{equation}

Similarly, assuming that the ansatz \eqref{eqn-varphi-dim-3-2} on $S^1 \times X^6$ is preserved, by the result of Lemma \ref{lem-Hodge-Laplacian-varphi-flow-dim-3} we see that the Laplacian flow leads to the evolution equation
\begin{equation}
\label{eqn-Laplacian-flow-dim-3-3}
    \frac{d}{dt} \Big( \Re(\Omega) - dr \wedge \omega \Big) = 2 \cal{L}_{\nabla_{(g_6)} (|\Omega|_\omega^{-\frac{2}{3}})} \Big( -\Re(\Omega) - 2 dr \wedge \omega \Big).
\end{equation}
The terms involving $\omega$ and $\Omega$ in both cases are similar and to tackle both these cases simultaneously, we let $h$ denote either the metric $g_4$ on $X^4$ or the metric $g_6$ on $X^6$. 

Motivated by the above and noting time dependencies, we consider ans\"{a}tze $(\omega_t,\Omega_t)$ on $X^{2n}$ that satisfy the coupled differential equations:
\begin{equation}
\label{eqn-ansatz-omega-flow}
    \frac{d \omega_t}{d t} = 2K\cal{L}_{\nabla_{h_t}  (|\Omega_t|_{\omega_t}^{-\frac{2}{3}})} \omega_t,
\end{equation}
\begin{equation}
\label{eqn-ansatz-Omega-flow}
    \frac{d \Omega_t}{d t} = -K\cal{L}_{\nabla_{h_t}  (|\Omega_t|_{\omega_t}^{-\frac{2}{3}})} \Omega_t,
\end{equation}
where the constant $K = 2^{\frac{n}{3}}$.

We note that if $(\omega_t,\Omega_t)$ satisfy \eqref{eqn-ansatz-omega-flow} and \eqref{eqn-ansatz-Omega-flow}, then they also satisfy \eqref{eqn-Laplacian-flow-dim-2-3} or \eqref{eqn-Laplacian-flow-dim-3-3} respectively. 

\begin{rmk}
\label{rmk-induced-metric-and-cplx-struct}
We also note that the metric $h_t$ along the flow is determined from the pair $(\omega_t,\Omega_t)$: from $\Omega_t$, we obtain the complex structure $J_t$ by defining the subbundle $T^{1,0}X \subset T_{\mathbb{C}} X$ to be the kernel of $\Omega_t$, and then defining $J_t$ by $+i$ on $T^{1,0}X$ and $-i$ on its conjugate $T^{0,1}X$. Then we define as usual $h_t(X,Y) = \omega_t(X, J_tY)$.
\end{rmk}

\begin{rmk}
\label{rmk-compatibility}
A priori, it is not known if the structures $(\omega_t, \Omega_t)$ along the flow will remain compatible and integrable for all time. However, the solutions presented in the sequel will satisfy any compatibility conditions required as they are obtained by pulling back compatible structures via diffeomorphisms.
\end{rmk}

In the next section, we will construct a solution $(\omega_t,\Omega_t)$ compatible with an integrable complex structure $J_t$ satisfying \eqref{eqn-ansatz-omega-flow} and \eqref{eqn-ansatz-Omega-flow}. Before this, in part to motivate the solution, we take a closer look at the Lie derivative term in the first of the coupled equations. We have the identity
\begin{equation}
\label{eqn-Lie-deriv-MA1/3-flow}
    \cal{L}_{\nabla_h (|\Omega|_{\omega}^{-\frac{2}{3}})} \omega = 2i \partial_J \overline{\partial}_J (|\Omega|_{\omega}^{-\frac{2}{3}}),
  \end{equation}
which holds on any K\"ahler Calabi--Yau structure $(X,J,\omega,h,\Omega)$. To show this, we first note that in local complex coordinates, we have
\begin{equation}
\label{eqn-grad-norm-local}
    \nabla_h (|\Omega|_{\omega}^{-\frac{2}{3}}) = \frac{\del}{\del z^j} (|\Omega|_{\omega}^{-\frac{2}{3}}) h^{j\overline{k}} \frac{\del}{\del \overline{z}^k} + \frac{\del}{\del \overline{z}^j} (|\Omega|_{\omega}^{-\frac{2}{3}}) h^{k \overline{j}} \frac{\del}{\del z^k}.
\end{equation}

Using our local expressions \eqref{eqn-omega-local-dim-2} and \eqref{eqn-omega-local-dim-3} for $\omega$ and Cartan's magic formula \eqref{eqn-Cartan's-magic-formula}, we see that
\begin{equation*}
\begin{aligned}
    \cal{L}_{\nabla_h (|\Omega|_{\omega}^{-\frac{2}{3}})} \omega &= d \Big[ \nabla_h (|\Omega|_{\omega}^{-\frac{2}{3}}) \hook \omega \Big] \\
    &= d \Big[ \Big(\frac{\del}{\del z^j} (|\Omega|_{\omega}^{-\frac{2}{3}}) h^{j\overline{k}} \frac{\del}{\del \overline{z}^k} + \frac{\del}{\del \overline{z}^j} (|\Omega|_{\omega}^{-\frac{2}{3}}) h^{k \overline{j}} \frac{\del}{\del z^k}\Big) \hook \Big(i h_{p\overline{q}} dz^p \wedge d\overline{z}^q\Big) \Big] \\
    &= -i d\Big[ \frac{\del}{\del z^j} (|\Omega|_{\omega}^{-\frac{2}{3}}) dz^j \Big] + i d\Big[ \frac{\del}{\del \overline{z}^j} (|\Omega|_{\omega}^{-\frac{2}{3}}) d\overline{z}^j \Big] \\
    &= 2i \partial \overline{\partial} (|\Omega|_{\omega}^{-\frac{2}{3}}).
\end{aligned}
\end{equation*}
The identity \eqref{eqn-Lie-deriv-MA1/3-flow} indicates that the flow of $\omega_t$ \eqref{eqn-ansatz-omega-flow} is related to the MA$^{\frac{1}{3}}$ flow; this will be made precise in the following section.

A similar consideration of the other Lie derivative term tells us that the complex structure $J_t$ on $X^4$ (and respectively on $X^6$) must be changing in time. This is because the expression\begin{equation*}
    \cal{L}_{\nabla_{h_t} (|\Omega_t|_{\omega_t}^{-\frac{2}{3}})} \Omega_t
\end{equation*}
produces terms of type $(1,1)$ (respectively type $(2,1)$) with respect to the current complex structure $J_t$. In order for $\Omega_t$ to remain a $(2,0)$-form (respectively $(3,0)$-form) so that we can define a time-dependent $G_2$-structure $\varphi_t$ by the expression
\begin{equation}
\label{eqn-ansatz-varphi-flow-2}
    \varphi_t = 
    \begin{cases}
        - dr^1 \wedge dr^2 \wedge dr^3 + dr^1 \wedge \omega_t \\
        \qquad + dr^2 \wedge \Re (\Omega_t) + dr^3 \wedge \Im (\Omega_t) &\text{ on } T^3 \times X^4, \\[5pt]
        \Re (\Omega_t) - dr \wedge \omega_t &\text{ on } S^1 \times X^6,
    \end{cases}
\end{equation}
on $M^7$, the complex structure $J_t$ must change as well. To solve the coupled system \eqref{eqn-ansatz-omega-flow}, \eqref{eqn-ansatz-Omega-flow}, we will act on compatible structures by a moving family of diffeomorphisms $\Theta_t$ so that from this new point of view, the complex structure is fixed; this idea can be found in \cite{FPPZ21a}.

\subsection{A Solution from the \texorpdfstring{MA$^{\frac{1}{3}}$}{MA1/3} Flow}
\label{subsect-soln-MA1/3}

Recall that $(X^{2n},\omega,J,h,\Omega)$ denotes a compact K\"ahler Calabi--Yau $n$-fold. Let $u_t$ be a smooth solution to the MA$^\frac{1}{3}$ flow
\begin{equation}
\label{eqn-MA1/3-flow-3}
    \frac{du_t}{dt} = 6K \Big(e^{-2 \log |\Omega|_{\omega}} \frac{\det (\omega + i \partial \overline{\partial} u_t)}{\det \omega} \Big)^\frac{1}{3}, \quad \omega + i \partial \overline{\partial} u_t > 0
\end{equation}
on $X^{2n}$ with $u_0 = 0$ and $K  = 2^{\frac{n}{3}}$ as before. It was proved in \cite{PZ19} that a solution to this flow exists for all time $t \in [0,\infty)$. We will discuss this in further detail in a later section \S \ref{sect-anal-prelims}. From the flow, we may define a family of K\"{a}hler metrics $\wtilde{\omega}_t = \omega + i \partial \overline{\partial} u_t$ on $X^{2n}$ defined for all $t \in [0,\infty)$. The complex structure $J$ is fixed along this flow, hence we have a family of K\"{a}hler triples $(\wtilde{\omega}_t, J, \wtilde{h}_t)$ on $X^{2n}$.

Using \eqref{eqn-MA1/3-flow-3} and \eqref{eqn-norm-Omega-dim-2} or \eqref{eqn-norm-Omega-dim-3}, we can compute that
\begin{equation*}
\begin{aligned}
    \frac{du_t}{dt} &= 6K \Big( \frac{\det (\omega + i \partial \overline{\partial} u_t)}{|\Omega|_\omega^2 \det \omega} \Big)^\frac{1}{3} \\
    &= 6K \Big( \frac{\det (\omega + i \partial \overline{\partial} u_t)}{|f|^2} \Big)^\frac{1}{3} \\
    &= 6K (|\Omega|_{\wtilde{\omega}_t}^{-\frac{2}{3}}).
\end{aligned}
\end{equation*}
Hence
\begin{equation}
\label{eqn-del-delbar-norm-Omega}
    \frac{d\wtilde{\omega}_t}{dt} = \frac{d}{dt}(\omega + i \partial \overline{\partial} u_t) = i \partial \overline{\partial} \frac{du_t}{dt} = 6Ki \partial \overline{\partial} (|\Omega|_{\wtilde{\omega}_t}^{-\frac{2}{3}}).
\end{equation}

Using the holomorphic volume form $\Omega$ and the smooth solution $\wtilde{\omega}_t$, we can define a time-dependent vector field $Y$ by
\begin{equation}
\label{eqn-Y_t-flow}
    Y_t = -K \nabla_{\wtilde{h}_t} (|\Omega|_{\wtilde{\omega}_t}^{-\frac{2}{3}}).
\end{equation}

Let $\Theta_t$ be the $1$-parameter family of diffeomorphisms generated by the vector field $Y$ in the sense that
\begin{equation}
\label{eqn-Theta_t-flow}
    \frac{d}{dt}\Theta_t(p) = Y_t (\Theta_t (p)), \quad \Theta_0 = \id_{X^6}.
\end{equation}
This family exists for all time $t \in [0,\infty)$ (see Lemma 3.15 in \cite{CK04}).

We can pullback our tensors of interest via this family of diffeomorphisms. Define
\begin{equation}
\label{eqn-omega_t-Omega_t-flow}
    \omega_t = \Theta_t^* \wtilde{\omega}_t, \quad \Omega_t = \Theta_t^* \Omega.
\end{equation}
In general $\Omega_t$ will not remain a $(3,0)$-form with respect to the original complex structure $J$ on $X^6$. However, by pulling $J$ back by the same diffeomorphism $\Theta_t$, we get a flow of complex structures $J_t = \Theta_t^* J$ which keeps $\Omega_t$ as a holomorphic volume form. Further, since each of our tensors were obtained via pullbacks, we also obtain an Riemannian metric $h_t = \Theta_t^* \wtilde{h}_t$ compatible with $\omega_t$ and $J_t$. That is, we have defined yet another family of K\"{a}hler triples $(\omega_t, J_t, h_t)$ on $X^6$.

Using DeTurck's trick and tensorial properties, we will show that the pair $(\omega_t,\Omega_t)$ is a solution to the coupled equations \eqref{eqn-ansatz-omega-flow} and \eqref{eqn-ansatz-Omega-flow}. Recall that the complex structure $J_t$ and in turn the metric $h_t$ are determined by the pair $(\omega_t,\Omega_t)$ and satisfy K\"{a}hler compatibility conditions as described above.

From \eqref{eqn-del-delbar-norm-Omega} we compute
\begin{equation*}
\begin{aligned}
    \frac{d\omega_t}{dt} &= \frac{d}{dt} (\Theta_t^* \wtilde{\omega}_t) \\
    &= \Theta_t^* (\cal{L}_{Y_t} \wtilde{\omega}_t) + \Theta_t^* \Big(\frac{d\wtilde{\omega}_t}{dt}\Big) \\
    &= \cal{L}_{(\Theta_t^{-1})_* Y_t} (\Theta_t^* \wtilde{\omega}_t) + \Theta_t^* \Big(6K i \partial \overline{\partial} (|\Omega|_{\wtilde{\omega}_t}^{-\frac{2}{3}}) \Big) \\
    &= \cal{L}_{-K(\Theta_t^{-1})_* [\nabla_{\wtilde{h}_t} (|\Omega|_{\wtilde{\omega}_t}^{-\frac{2}{3}})]} \omega_t + 6K i \partial_t \overline{\partial}_t \Big( \Theta_t^*(|\Omega|_{\wtilde{\omega}_t}^{-\frac{2}{3}}) \Big) \\
    &= -K\cal{L}_{\nabla_{h_t} (|\Omega_t|_{\omega_t}^{-\frac{2}{3}})} \omega_t + 6K i \partial_t \overline{\partial}_t (|\Omega_t|_{\omega_t}^{-\frac{2}{3}}).
\end{aligned}
\end{equation*}

Using \eqref{eqn-Lie-deriv-MA1/3-flow}, we see that
\begin{equation*}
    \frac{d\omega_t}{dt} = 2K\cal{L}_{\nabla_{h_t} (|\Omega_t|_{\omega_t}^{-\frac{2}{3}})} \omega_t
\end{equation*}
which is just \eqref{eqn-ansatz-omega-flow}.

Similarly, we can check that
\begin{equation*}
\begin{aligned}
    \frac{d\Omega_t}{dt} &= \frac{d}{dt} (\Theta_t^* \Omega) \\
    &= \Theta_t^* (\cal{L}_{Y_t} \Omega) \\
    &= \cal{L}_{(\Theta_t^{-1})_* Y_t} (\Theta_t^* \Omega) \\
    &= \cal{L}_{-K(\Theta_t^{-1})_* [\nabla_{\wtilde{h}_t} (|\Omega|_{\wtilde{\omega}_t}^{-\frac{2}{3}})]} \Omega_t \\
    &= -K\cal{L}_{\nabla_{h_t} (|\Omega_t|_{\omega_t}^{-\frac{2}{3}})} \Omega_t,
\end{aligned}
\end{equation*}
which is \eqref{eqn-ansatz-Omega-flow}.

Combining what we have thus far, we see that if we start the $G_2$-Laplacian flow with initial data of the form
\begin{equation}
\label{eqn-varphi-flow-initial-2}
    \varphi = 
    \begin{cases}
        - dr^1 \wedge dr^2 \wedge dr^3 + dr^1 \wedge \omega \\
        \qquad + dr^2 \wedge \Re (\Omega) + dr^3 \wedge \Im (\Omega) &\text{ on } M^7 = T^3 \times X^4, \\[5pt]
        \Re (\Omega) - dr \wedge \omega &\text{ on } M^7 = S^1 \times X^6.
    \end{cases}
\end{equation}
then the $G_2$-structures induced by the MA$^{\frac{1}{3}}$ flow and given by
\begin{equation}
\label{eqn-ansatz-varphi-flow-3}
    \varphi_t = 
    \begin{cases}
        - dr^1 \wedge dr^2 \wedge dr^3 + dr^1 \wedge \Theta_t^* \wtilde{\omega}_t \\
        \qquad + dr^2 \wedge \Re (\Theta_t^* \Omega) + dr^3 \wedge \Im (\Theta_t^* \Omega) &\text{ on } M^7 = T^3 \times X^4, \\[5pt]
        \Re (\Theta_t^* \Omega) - dr \wedge \Theta_t^* \wtilde{\omega}_t &\text{ on } M^7 = S^1 \times X^6.
    \end{cases} 
\end{equation}
solve the $G_2$-Laplacian flow equation \eqref{eqn-Laplacian-flow-2} for all time $t \in [0,\infty)$. By uniqueness of the flow \cite{BX04}, we conclude that the $G_2$-Laplacian flow preserves the ansatz \eqref{eqn-ansatz-varphi-flow-3} and is equivalent to the MA$^{\frac{1}{3}}$ flow for this class of initial data.

\section{Laplacian Coflow of \texorpdfstring{$G_2$}{G2}-Structures}
\label{sect-Laplacian-Coflow}

We apply a similar technique to the Laplacian coflow of $G_2$-structures. This flow was first introduced in \cite{KMT12} and, alongside a modification of it, has been studied in \cite{Gri13, Gri16, LSES22}.

\begin{defn}
\label{defn-Laplacian-coflow}
A time-dependent $G_2$-structure $\varphi_t$ on a $7$-dimensional manifold $M$ defined on some interval $[0,T)$ satisfies the Laplacian coflow equation if
\begin{equation}
\label{eqn-Laplacian-coflow-2}
    \frac{d\psi_t}{dt} = {\Delta_d}_t \psi_t
\end{equation}
wherever $\psi_t$ is defined. Here $\psi_t = \star_t \varphi_t$ is the Hodge dual of $\varphi_t$ and ${\Delta_d}_t = dd^*_t + d^*_td$ is the Hodge Laplacian with respect to the metric ${(g_7)}_t$ and volume form ${(\vol_7)}_t$ induced from $\varphi_t$.
\end{defn}

\begin{rmk}
\label{rmk-coflow-sign}
We note that the Laplacian coflow was originally introduced with a minus sign on the right-hand side of \eqref{eqn-Laplacian-coflow-2} by analogy with the heat equation but we will follow the convention in \cite{Lot20}. We also remark that the $4$-form $\psi$ does not determine an orientation, however we may assume an initial orientation which will stay fixed along the flow.
\end{rmk}

Unlike the Laplacian flow, it is not known whether we have short-time existence or uniqueness even when starting with an initially coclosed structure. We do know that the flow does preserve coclosedness when starting with a coclosed structure in the same way that the Laplacian flow preserves closedness \cite{Lot20}. Additionally, we have that the stationary points of the coflow are torsion-free $G_2$-structures.

\subsection{The Hodge Laplacian of \texorpdfstring{$\psi$}{psi} on \texorpdfstring{$T^3 \times X^4$}{T3 x X4}}
\label{subsect-Hodge-Laplacian-psi-dim-2}

As we did in \S \ref{subsect-Hodge-Laplacian-varphi-dim-2}, we make appropriate choices for the functions $F$ and $G$ in \eqref{eqn-varphi-dim-2} and study the resulting equations. In order to get a coclosed structure, we pick $F = 1$ and $G = |\Omega|_\omega$ and define the $G_2$-structure on $M^7 = T^3 \times X^4$ by
\begin{equation}
\label{eqn-varphi-dim-2-3}
    \varphi = -|\Omega|_\omega dr^1 \wedge dr^2 \wedge dr^3 + dr^1 \wedge |\Omega|_\omega \omega + dr^2 \wedge \Re \Big( \frac{1}{|\Omega|_\omega} \Omega \Big) + dr^3 \wedge \Im \Big( \frac{1}{|\Omega|_\omega} \Omega \Big).
\end{equation}

The identities \eqref{eqn-g_7-dim-3}, \eqref{eqn-vol_7-dim-3}, and \eqref{eqn-psi-dim-3} tell us that the metric $g_7$ on $M^7$ is 
\begin{equation}
\label{eqn-g_7-dim-2-3}
    g_7 = 2^{\frac{4}{3}} |\Omega|_\omega^2 (dr^1)^2 + 2^{-\frac{2}{3}} (dr^2)^2 + 2^{-\frac{2}{3}} (dr^3)^2 + 2^{-\frac{2}{3}} g_4,
\end{equation}
the volume form $\vol_7$ is
\begin{equation}
\label{eqn-vol_7-dim-2-3}
    \vol_7 = 2^{-\frac{4}{3}} |\Omega|_\omega dr^1 \wedge dr^2 \wedge dr^3 \wedge \vol_4,
\end{equation}
and the dual $4$-form $\psi$ is
\begin{equation}
\label{eqn-psi-dim-2-3}
    \psi = -2^{-\frac{4}{3}} \vol_4 + 2^{-\frac{4}{3}} dr^2 \wedge dr^3 \wedge \omega + 2^{\frac{2}{3}} dr^3 \wedge dr^1 \wedge \Re (\Omega) + 2^{\frac{2}{3}} dr^1 \wedge dr^2 \wedge \Im (\Omega).
\end{equation}

We also have the following identities for the Hodge star $\star_7$ on $M^7$ when $\alpha \in \Omega^k(X^4)$:
\begin{equation}
\label{eqn-star_7-dim-2-3}
\begin{aligned}
    \star_7 \alpha &= (-1)^k 2^{(-\frac{4}{3} + \frac{2}{3}k)} |\Omega|_\omega dr^1 \wedge dr^2 \wedge dr^3 \wedge \star_4 \alpha, \\
    \star_7 (dr^1 \wedge \alpha) &= 2^{(-\frac{8}{3} + \frac{2}{3}k)} |\Omega|_\omega^{-1} dr^2 \wedge dr^3 \wedge \star_4 \alpha, \\
    \star_7 (dr^2 \wedge \alpha) &= 2^{(-\frac{2}{3} + \frac{2}{3}k)} |\Omega|_\omega dr^3 \wedge dr^1 \wedge \star_4 \alpha, \\
    \star_7 (dr^3 \wedge \alpha) &= 2^{(-\frac{2}{3} + \frac{2}{3}k)} |\Omega|_\omega dr^1 \wedge dr^2 \wedge \star_4 \alpha, \\
    \star_7 (dr^1 \wedge dr^2 \wedge \alpha) &= (-1)^k 2^{(-2 + \frac{2}{3}k)} |\Omega|_\omega^{-1} dr^3 \wedge \star_4 \alpha, \\
    \star_7 (dr^3 \wedge dr^1 \wedge \alpha) &= (-1)^k 2^{(-2 + \frac{2}{3}k)} |\Omega|_\omega^{-1} dr^2 \wedge \star_4 \alpha, \\
    \star_7 (dr^2 \wedge dr^3 \wedge \alpha) &= (-1)^k 2^{\frac{2}{3}k} |\Omega|_\omega dr^1 \wedge \star_4 \alpha, \\
    \star_7 (dr^1 \wedge dr^2 \wedge dr^3 \wedge \alpha) &= 2^{(-\frac{4}{3} + \frac{2}{3}k)} |\Omega|_\omega^{-1} \star_4 \alpha.
\end{aligned}
\end{equation}

As before, we are interested in the evolution equations given by the Laplacian coflow with initial data \eqref{eqn-varphi-dim-2-3}. As such, we compute the Hodge Laplacian on $\psi$.

\begin{lem}
\label{lem-Hodge-Laplacian-psi-coflow-dim-2}
If $\varphi$ is the $G_2$-structure defined by \eqref{eqn-varphi-dim-2-3}, then
\begin{equation}
\label{eqn-Hodge-Laplacian-psi-coflow-dim-2}
    \Delta_d \psi = 2^{\frac{2}{3}} \cal{L}_{\nabla_{(g_4)} (\log |\Omega|_\omega)} \Big( 2^{-\frac{4}{3}} \vol_4 - 2^{-\frac{4}{3}} dr^2 \wedge dr^3 \wedge \omega + 2^{\frac{2}{3}}  dr^3 \wedge dr^1 \wedge \Re (\Omega) + 2^{\frac{2}{3}} dr^1 \wedge dr^2 \wedge \Im (\Omega) \Big).
\end{equation}
\end{lem}

\begin{proof}
Since $\psi$ is closed, we have
\begin{equation*}
    \Delta_d \psi = dd^*\psi + d^*d\psi = dd^*\psi.
\end{equation*}
Now, $\psi$ is a $4$-form on the $7$-dimensional manifold $M^7$, and so
\begin{equation*}
    dd^*\psi = d\star_7 d\star_7 \psi = d\star_7 d\varphi.
\end{equation*}

We compute similarly to the proof of Lemma \ref{lem-Hodge-Laplacian-varphi-flow-dim-2} that
\begin{equation*}
\begin{aligned}
    d\varphi &= d \Big( -|\Omega|_\omega dr^1 \wedge dr^2 \wedge dr^3 + dr^1 \wedge |\Omega|_\omega \omega \\
    &\qquad \qquad + dr^2 \wedge \Re \Big( \frac{1}{|\Omega|_\omega} \Omega \Big) + dr^3 \wedge \Im \Big( \frac{1}{|\Omega|_\omega} \Omega \Big) \Big) \\
    &= -|\Omega|_\omega d(\log |\Omega|_\omega) \wedge dr^1 \wedge dr^2 \wedge dr^3 + |\Omega|_\omega d(\log |\Omega|_\omega) \wedge dr^1 \wedge \omega \\
    &\qquad - |\Omega|_\omega^{-1} d(\log |\Omega|_\omega) \wedge dr^2 \wedge \Re (\Omega) - |\Omega|_\omega^{-1} d(\log |\Omega|_\omega) \wedge dr^3 \wedge \Im (\Omega).
\end{aligned}
\end{equation*}

Again, using the identities \eqref{eqn-hook-star} and \eqref{eqn-grad-flat}, we get
\begin{equation*}
\begin{aligned}
    \star_7 d\varphi &= -|\Omega|_\omega \star_7 \Big[ d(\log |\Omega|_\omega) \wedge dr^1 \wedge dr^2 \wedge dr^3 \Big] + |\Omega|_\omega \star_7 \Big[ d(\log |\Omega|_\omega) \wedge dr^1 \wedge \omega \Big] \\
    &\qquad - |\Omega|_\omega^{-1} \star_7 \Big[ d(\log |\Omega|_\omega \wedge dr^2 \wedge \Re (\Omega) \Big]  - |\Omega|_\omega^{-1} \star_7 \Big[ d(\log |\Omega|_\omega) \wedge dr^3 \wedge \Im (\Omega) \Big] \\
    &= |\Omega|_\omega \Big( d(\log |\Omega|_\omega) \Big)^{\sharp_7} \hook \star_7 \Big[ dr^1 \wedge dr^2 \wedge dr^3 \Big] - |\Omega|_\omega \Big( d(\log |\Omega|_\omega) \Big)^{\sharp_7} \hook \star_7 \Big[ dr^1 \wedge \omega \Big] \\
    &\qquad + |\Omega|_\omega^{-1} \Big( d(\log |\Omega|_\omega) \Big)^{\sharp_7} \hook \star_7 \Big[ dr^2 \wedge \Re (\Omega) \Big] + |\Omega|_\omega^{-1} \Big( d(\log |\Omega|_\omega) \Big)^{\sharp_7} \hook \star_7 \Big[ dr^3 \wedge \Im (\Omega) \Big] \\
    &= |\Omega|_\omega \Big( 2^{\frac{2}{3}} \nabla_{(g_4)} (\log |\Omega|_\omega) \Big) \hook \Big[ 2^{-\frac{4}{3}} |\Omega|_\omega^{-1} \vol_4 \Big] - |\Omega|_\omega \Big( 2^{\frac{2}{3}} \nabla_{(g_4)} (\log |\Omega|_\omega) \Big) \hook \Big[ 2^{-\frac{4}{3}} |\Omega|_\omega^{-1} dr^2 \wedge dr^2 \wedge \omega \Big] \\
    &\qquad + |\Omega|_\omega^{-1} \Big( 2^{\frac{2}{3}} \nabla_{(g_4)} (\log |\Omega|_\omega) \Big) \hook \Big[ 2^{\frac{2}{3}} |\Omega|_\omega dr^3 \wedge dr^1 \wedge \Re (\Omega) \Big] \\
    &\qquad + |\Omega|_\omega^{-1} \Big( 2^{\frac{2}{3}} \nabla_{(g_4)} (\log |\Omega|_\omega) \Big) \hook \Big[ 2^{\frac{2}{3}} |\Omega|_\omega dr^1 \wedge dr^2 \wedge \Im (\Omega) \Big] \\
    &= 2^{\frac{2}{3}} \Big( \nabla_{(g_4)} (\log |\Omega|_\omega) \Big) \hook \Big[ 2^{-\frac{4}{3}} \vol_4 - 2^{-\frac{4}{3}} dr^2 \wedge dr^3 \wedge \omega + 2^{\frac{2}{3}} dr^3 \wedge dr^1 \wedge \Re (\Omega) + 2^{\frac{2}{3}} dr^1 \wedge dr^2 \wedge \Im (\Omega) \Big].
\end{aligned}
\end{equation*}

Cartan's magic formula \eqref{eqn-Cartan's-magic-formula} then yields
\begin{equation*}
    \Delta_d \psi = d\star_7 d\varphi = 2^{\frac{2}{3}} \cal{L}_{\nabla_{(g_4)} (\log |\Omega|_\omega)} \Big( 2^{-\frac{4}{3}} \vol_4 - 2^{-\frac{4}{3}} dr^2 \wedge dr^3 \wedge \omega + 2^{\frac{2}{3}}  dr^3 \wedge dr^1 \wedge \Re (\Omega) + 2^{\frac{2}{3}} dr^1 \wedge dr^2 \wedge \Im (\Omega) \Big),
\end{equation*}
as desired.
\end{proof}

Using the intermediate expressions in the above proof, we can compute the torsion forms of the $G_2$-structure $\varphi$.

\begin{lem}
\label{lem-torsion-coflow-dim-2}
If $\varphi$ is the $G_2$-structure defined by \eqref{eqn-varphi-dim-2-3}, then the torsion forms are given by
\begin{equation}
\label{eqn-torsion-coflow-dim-2}
\begin{aligned}
    &\tau_0 = 0,\quad \tau_1 = 0,\quad \tau_2 = 0, \\
    &\tau_3 = 2^{\frac{2}{3}} \Big( \nabla_{(g_4)} (\log |\Omega|_\omega) \Big) \hook \Big[ 2^{-\frac{4}{3}} \vol_4 - 2^{-\frac{4}{3}} dr^2 \wedge dr^3 \wedge \omega + 2^{\frac{2}{3}} dr^3 \wedge dr^1 \wedge \Re (\Omega) + 2^{\frac{2}{3}} dr^1 \wedge dr^2 \wedge \Im (\Omega) \Big].
\end{aligned}
\end{equation}
\end{lem}

\begin{proof}
Since $\varphi$ is coclosed, we have $d\psi = 0$. If follows from \eqref{eqn-torsion} that $\tau_1$ and $\tau_2$ both vanish. Next, using our previous expression for $d\varphi$, we have
\begin{equation*}
\begin{aligned}
    \varphi \wedge d\varphi &= \Big[ -|\Omega|_\omega dr^1 \wedge dr^2 \wedge dr^3 + dr^1 \wedge |\Omega|_\omega \omega + dr^2 \wedge \Re \Big( \frac{1}{|\Omega|_\omega} \Omega \Big) + dr^3 \wedge \Im \Big( \frac{1}{|\Omega|_\omega} \Omega \Big) \Big] \\
    &\qquad \wedge \Big[-|\Omega|_\omega d(\log |\Omega|_\omega) \wedge dr^1 \wedge dr^2 \wedge dr^3 + |\Omega|_\omega d(\log |\Omega|_\omega) \wedge dr^1 \wedge \omega \\
    &\qquad \qquad  - |\Omega|_\omega^{-1} d(\log |\Omega|_\omega) \wedge dr^2 \wedge \Re (\Omega) - |\Omega|_\omega^{-1} d(\log |\Omega|_\omega) \wedge dr^3 \wedge \Im (\Omega) \Big] \\
    &= 0
\end{aligned}
\end{equation*}
since $\log |\Omega|_\omega$ does not depend on the angle coordinates $r^1, r^2$, and $r^3$. From \eqref{eqn-torsion-id}, we see that $\tau_0 = 0$. Once again from \eqref{eqn-torsion}, we see that
\begin{equation*}
\begin{aligned}
    \tau_3 &= (\star_7)^2 \tau_3 = \star_7 d\varphi \\
    &=  2^{\frac{2}{3}} \Big( \nabla_{(g_4)} (\log |\Omega|_\omega) \Big) \hook \Big[ 2^{-\frac{4}{3}} \vol_4 - 2^{-\frac{4}{3}} dr^2 \wedge dr^3 \wedge \omega + 2^{\frac{2}{3}} dr^3 \wedge dr^1 \wedge \Re (\Omega) + 2^{\frac{2}{3}} dr^1 \wedge dr^2 \wedge \Im (\Omega) \Big].
\end{aligned}
\end{equation*}
\end{proof}

\subsection{The Hodge Laplacian of \texorpdfstring{$\psi$}{psi} on \texorpdfstring{$S^1 \times X^6$}{S1 x X6}}
\label{subsect-Hodge-Laplacian-psi-dim-3}

We return to the $3$-fold setting from \S \ref{subsect-CY-3-Folds}. By choosing $F = 1$ and $G = |\Omega|_\omega$ in \eqref{eqn-varphi-dim-3}, we get the $G_2$-structure on $M^7 = S^1 \times X^6$ by
\begin{equation}
\label{eqn-varphi-dim-3-3}
    \varphi = \Re \Big( \frac{1}{|\Omega|_\omega} \Omega \Big) - dr \wedge |\Omega|_\omega \omega.
\end{equation}

Using \eqref{eqn-g_7-dim-3}, \eqref{eqn-vol_7-dim-3}, and \eqref{eqn-psi-dim-3}, we have that the metric $g_7$ on $M^7$ is
\begin{equation}
\label{eqn-g_7-dim-3-3}
    g_7 = 4|\Omega|_\omega^2 dr^2 + \frac{1}{2} g_6,
\end{equation}
the volume form $\vol_7$ is
\begin{equation}
\label{eqn-vol_7-dim-3-3}
    \vol_7 = \frac{1}{4}|\Omega|_\omega dr \wedge \vol_6,
\end{equation}
and the dual $4$-form $\psi$ is
\begin{equation}
\label{eqn-psi-dim-3-3}
    \psi = - 2dr \wedge \Im(\Omega) - \frac{1}{4} \frac{1}{2} \omega^2.
\end{equation}

We also have the following expressions for the Hodge star operator $\star_7$ from \eqref{eqn-star_7-dim-3}:
\begin{equation}
\label{eqn-star_7-3}
\begin{aligned}
    \star_7 \beta &= (-1)^k 2^{(-2+k)} |\Omega|_\omega dr \wedge \star_6 \beta, \\
    \star_7 (dr \wedge \beta) &= 2^{(-4+k)} |\Omega|_\omega^{-1} \star_6 \beta,
\end{aligned}
\end{equation}
where $\beta \in \Omega^k(X^6)$ is a $k$-form on $X$.

Both $\omega$ and $\Omega$ are closed, so $\varphi$ is a coclosed $G_2$-structure. We now compute the Hodge Laplacian of $\psi$.

\begin{lem}
\label{lem-Hodge-Laplacian-psi-coflow-dim-3}
If $\varphi$ is the $G_2$-structure defined by \eqref{eqn-varphi-dim-3-3}, then
\begin{equation}
\label{eqn-Hodge-Laplacian-psi-coflow-dim-3}
    \Delta_d \psi = 2 \cal{L}_{\nabla (\log |\Omega|_\omega)} \Big( -2 dr \wedge \Im(\Omega) + \frac{1}{4} \frac{1}{2}\omega^2 \Big).
\end{equation}
\end{lem}

\begin{proof}
As before, we have
\begin{equation*}
    \Delta_d \psi = dd^* \psi = d \star_7 d \star_7 \psi = d \star_7 d\varphi.
\end{equation*}

We compute similarly to earlier that
\begin{equation*}
\begin{aligned}
    d\varphi &= d \Big( \Re\Big(\frac{1}{|\Omega|_\omega} \Omega\Big) - dr \wedge |\Omega|_\omega \omega\Big) \\
    &= - |\Omega|_\omega^{-1} d(\log |\Omega|_\omega) \wedge \Re (\Omega) - |\Omega|_\omega d(\log |\Omega|_\omega) \wedge dr \wedge \omega.
\end{aligned}
\end{equation*}

Taking the Hodge star of both sides and applying the identities \eqref{eqn-hook-star} and \eqref{eqn-grad-flat}, we get
\begin{equation*}
\begin{aligned}
    \star_7 d\varphi &= - |\Omega|_\omega^{-1} \star_7 \Big[ d(\log |\Omega|_\omega) \wedge \Re (\Omega) \Big] - |\Omega|_\omega \star_7 \Big[d(\log |\Omega|_\omega) \wedge dr \wedge \omega\Big] \\
    &=  |\Omega|_\omega^{-1} \Big(d(\log |\Omega|_\omega) \Big)^{\sharp_7} \hook \star_7 \Big[ \Re (\Omega) \Big] + |\Omega|_\omega \Big(d(\log |\Omega|_\omega) \Big)^{\sharp_7} \hook \star_7 \Big[ dr \wedge \omega\Big] \\
    &=  |\Omega|_\omega^{-1} \Big( 2 \nabla_{(g_6)} (\log |\Omega|_\omega) \Big) \hook \Big[ -2 |\Omega|_\omega dr \wedge \Im  (\Omega) \Big] + |\Omega|_\omega \Big( 2 \nabla_{(g_6)} (\log |\Omega|_\omega) \Big) \hook \Big[ \frac{1}{4} \frac{1}{2} \omega^2 \Big] \\
    &= 2\Big(\nabla_{(g_6)} (\log |\Omega|_\omega) \Big) \hook \Big[ -2dr \wedge \Im(\Omega) + \frac{1}{4} \frac{1}{2}\omega^2\Big].
\end{aligned}
\end{equation*}

Then applying Cartan's magic formula \eqref{eqn-Cartan's-magic-formula} and the fact that $\omega$ and $\Omega$ are closed, we get
\begin{equation*}
    \Delta_d \psi = d\star_7 d\varphi = 2\cal{L}_{\nabla_{(g_6)} (\log |\Omega|_\omega)} \Big( - 2dr \wedge \Im(\Omega) + \frac{1}{4} \frac{1}{2} \omega^2\Big)
\end{equation*}
as desired.
\end{proof}

Again, we can use the expressions obtained in the above proof to compute the torsion forms of $\varphi$.

\begin{lem}
\label{lem-torsion-coflow-dim-3}
If $\varphi$ is the $G_2$-structure defined by \eqref{eqn-varphi-dim-3-3}, then the torsion forms are given by
\begin{equation}
\label{eqn-torsion-coflow-dim-3}
    \tau_0 = 0,\quad \tau_1 = 0,\quad \tau_2 = 0, \tau_3 = 2 \Big(\nabla_{(g_6)} (\log |\Omega|_\omega)\Big) \hook \Big[ - 2dr \wedge \Im(\Omega) + \frac{1}{4} \frac{1}{2}\omega^2\Big].
\end{equation}
\end{lem}

\begin{proof}
Since $\varphi$ is coclosed, we have $d\psi = 0$. If follows from \eqref{eqn-torsion} that $\tau_1$ and $\tau_2$ both vanish. Next, using our previous expression for $d\varphi$, we have
\begin{equation*}
\begin{aligned}
    \varphi \wedge d\varphi &= \Big[ \Re\Big(\frac{1}{|\Omega|_\omega} \Omega\Big) - dr \wedge |\Omega|_\omega \omega\Big] \wedge \Big[- |\Omega|_\omega^{-1} d(\log |\Omega|_\omega) \wedge \Re (\Omega) - |\Omega|_\omega d(\log |\Omega|_\omega) \wedge dr \wedge \omega\Big] = 0
\end{aligned}
\end{equation*}
since $\Re(\Omega) \wedge \omega = 0$ due to type considerations and $\Re (\Omega) \wedge \Re (\Omega) = 0$ as $\Re (\Omega)$ is a $3$-form. From \eqref{eqn-torsion-id}, we see that $\tau_0 = 0$. Once again from \eqref{eqn-torsion}, we see that
\begin{equation*}
    \tau_3 = (\star_7)^2 \tau_3 = \star_7 d\varphi = 2\Big(\nabla_{(g_6)} (\log |\Omega|_\omega) \Big) \hook \Big[- 2dr \wedge \Im(\Omega) + \frac{1}{4} \frac{1}{2} \omega^2 \Big].
\end{equation*}
\end{proof}

\subsection{The Evolution Equations}
\label{subsect-evolution-eqns-coflow}

Substituting ansatz \eqref{eqn-varphi-dim-2-3} into the coflow, we see that Lemma \ref{lem-Hodge-Laplacian-psi-coflow-dim-2} yields the evolution equation
\begin{equation}
\label{eqn-Laplacian-coflow-dim-2-3}
\begin{aligned}\
    &\frac{d}{dt} \Big( - 2^{-\frac{4}{3}} \vol_4 + 2^{-\frac{4}{3}} dr^2 \wedge dr^3 \wedge \omega + 2^{\frac{2}{3}} dr^3 \wedge dr^1 \wedge \Re (\Omega) + 2^{\frac{2}{3}} dr^1 \wedge dr^2 \wedge \Im (\Omega) \Big) \\
    &= 2^{\frac{2}{3}} \cal{L}_{\nabla_{(g_4)} (\log |\Omega|_\omega)} \Big( 2^{-\frac{4}{3}} \vol_4 - 2^{-\frac{4}{3}} dr^2 \wedge dr^3 \wedge \omega + 2^{\frac{2}{3}} dr^3 \wedge dr^1 \wedge \Re (\Omega) + 2^{\frac{2}{3}} dr^1 \wedge dr^2 \wedge \Im (\Omega) \Big)
\end{aligned}
\end{equation}
on $T^3 \times X^4$.

Analogously, substituting ansatz \eqref{eqn-varphi-dim-3-3} into the coflow and using Lemma \ref{lem-Hodge-Laplacian-psi-coflow-dim-3} gives us the evolution equation
\begin{equation}
\label{eqn-Laplacian-coflow-dim-3-3}
    \frac{d}{dt} \Big( - 2dr \wedge \Im(\Omega) - \frac{1}{4} \frac{1}{2} \omega^2 \Big) = 2 \cal{L}_{\nabla (\log |\Omega|_\omega)} \Big( - 2dr \wedge \Im(\Omega) + \frac{1}{4} \frac{1}{2} \omega^2 \Big)
\end{equation}
on $M^7 = S^1 \times X^6$.

As they were in \S \ref{subsect-evolution-eqns-flow}, the terms involving $\omega$ and $\Omega$ in both equations are similar. We again work on both cases simultaneously, letting $h$ denote either the metric $g_4$ on $X^4$ or the metric $g_6$ on $X^6$.

By considering the terms separately and noting time dependencies, we consider ans\"{a}tze $(\omega_t,\Omega_t)$ on $X^{2n}$ satisfying
\begin{equation}
\label{eqn-ansatz-omega-coflow}
    \frac{d \omega_t}{d t} = -K \cal{L}_{\nabla_{h_t}  (\log |\Omega_t|_{\omega_t})} \omega_t,
\end{equation}
\begin{equation}
\label{eqn-ansatz-Omega-coflow}
    \frac{d \Omega_t}{d t} = K \cal{L}_{\nabla_{h_t}  (\log |\Omega_t|_{\omega_t})} \Omega_t,
\end{equation}
where again $K$ is the constant $2^{\frac{n}{3}}$.

As with the coupled equations in \S \ref{subsect-evolution-eqns-flow}, solutions to \eqref{eqn-ansatz-omega-coflow} and \eqref{eqn-ansatz-Omega-coflow} will satisfy \eqref{eqn-Laplacian-coflow-dim-2-3} or \eqref{eqn-Laplacian-coflow-dim-3-3} as well. The metric $h_t$ along the flow is determined from the pair $(\omega_t,\Omega_t)$ in the way noted in Remark \ref{rmk-induced-metric-and-cplx-struct}. Similarly, we have the same compatibility considerations as in Remark \ref{rmk-compatibility} in order for the equations to be well-defined. Since the solutions to be presented are again obtained via pullback of compatible structures by diffeomorphisms, the compatibility conditions are satisfied for all time $t \in [0,\infty)$.

Before presenting the solution, we compute the Lie derivative term in \eqref{eqn-ansatz-omega-coflow}. By a similar computation as identity \eqref{eqn-Lie-deriv-MA1/3-flow}, we see that
\begin{equation*}
\begin{aligned}
    \cal{L}_{\nabla_h (\log |\Omega|_\omega)} \omega &= 2i \partial \overline{\partial} (\log |\Omega|_\omega).
\end{aligned}
\end{equation*}

Using \eqref{eqn-norm-Omega-dim-2} and \eqref{eqn-norm-Omega-dim-3} where $f$ is a nowhere vanishing holomorphic function, we have
\begin{equation*}
\begin{aligned}
    2i \partial_t \overline{\partial}_t (\log |\Omega_t|_{\omega_t}) &= i \partial_t \overline{\partial}_t \Big( \log \frac{|f_t|^2}{\det {(h_t)}_{p\overline{q}}} \Big) \\
    &= - i \partial_t \overline{\partial}_t (\log \det {(h_t)}_{p\overline{q}}) + i \partial_t \overline{\partial}_t (\log |f_t|^2) \\
    &= \Ric (\omega_t, J_t) + 0.
\end{aligned}
\end{equation*}

Thus, we get that
\begin{equation}
\label{eqn-Lie-deriv-KR-flow}
    \cal{L}_{\nabla_{h_t} (\log |\Omega_t|_{\omega_t})} \omega_t = \Ric(\omega_t, J_t),
\end{equation}
which is reminiscent of the K\"{a}hler--Ricci flow when combined with \eqref{eqn-ansatz-omega-coflow}.

As in \S \ref{subsect-evolution-eqns-flow}, the other Lie derivative term produces terms of type $(1,1)$ with respect to the complex structure $J_t$ on $X^4$ (respectively, we get terms of type $(2,1)$ on $X^6$). We again conclude that the complex structure $J_t$ must move with respect to time. In order to solve the coupled system, we will again pullback compatible structures by a moving family of diffeomorphisms.

\subsection{A Solution from the K\"{a}hler--Ricci Flow}
\label{subsect-soln-KR}

Let $\wtilde{\omega}_t$ be the unique smooth solution for the (rescaled) K\"{a}hler--Ricci flow
\begin{equation}
\label{eqn-KR-flow-3}
    \frac{d\wtilde{\omega}_t}{dt} = -2K\Ric(\wtilde{\omega}_t, J), \quad \wtilde{\omega}_0 = \omega
\end{equation}
on $(X^{2n},J,\Omega)$ where $K > 0$ is a constant. The existence of the holomorphic volume form $\Omega$ implies that $c_1(X^{2n}) = 0$, and so this solution exists for all time $t \in [0,\infty)$ \cite{Cao85} (or e.g. \cite{SW13,Tos18} for a modern reference on K\"ahler-Ricci flow). We recall that the complex structure $J$ on $X^{2n}$ remains fixed and we obtain a family of K\"{a}hler triples $(\wtilde{\omega}_t, J, \wtilde{h}_t)$ on $X^6$.

We can use the holomorphic volume form $\Omega$ and the smooth solution $\wtilde{\omega}_t$ to define a time-dependent vector field
\begin{equation}
\label{eqn-Y_t-coflow}
    Y_t = K\nabla_{\wtilde{h}_t} (\log |\Omega|_{\wtilde{\omega}_t}).
\end{equation}

Let $\Theta_t$ denote the $1$-parameter family of diffeomorphisms generated by the vector field $Y$. That is, set
\begin{equation}
\label{eqn-Theta_t-coflow}
    \frac{d}{dt} \Theta_t(p) = Y_t(\Theta_t(p)), \quad \Theta_0 = \id_{X^{2n}}.
\end{equation}

As we did in \S \ref{subsect-soln-MA1/3}, we can pull $\wtilde{\omega}_t$ and $\Omega$ back by this family and define
\begin{equation}
\label{eqn-omega_t-Omega_t-coflow}
    \omega_t = \Theta_t^* \wtilde{\omega}_t, \quad \Omega_t = \Theta_t^* \Omega.
\end{equation}

The same subtleties as in \S \ref{subsect-soln-MA1/3} apply here, however, by pulling the complex structure $J$ and the flowing metric $\wtilde{h}_t$ by the same family $\Theta_t$, we obtain another family of K\"{a}hler triples $(\omega_t, J_t,h_t)$ on $X^{2n}$ where $J_t = \Theta_t^* J$ and $h_t = \Theta_t^* \wtilde{h}_t$. In particular, $\Omega_t$ is a holomorphic $(n,0)$-form with respect to the moving complex structure $J_t$.

With this setup, we apply DeTurck's trick and the diffeomorphism invariance of the Ricci tensor in order to show that $(\omega_t,\Omega_t)$ is a solution to the evolution equations \eqref{eqn-ansatz-omega-coflow} and \eqref{eqn-ansatz-Omega-coflow}. As noted above, the pair $(\omega_t,\Omega_t)$ satisfy the required compatibility conditons for well-definition of the involved structures.

By DeTurck's trick and \eqref{eqn-KR-flow-3}, we can check that
\begin{equation*}
\begin{aligned}
    \frac{d\omega_t}{dt} &= \frac{d}{dt} (\Theta_t^* \wtilde{\omega}_t) \\
    &= \Theta_t^* (\cal{L}_{Y_t} \wtilde{\omega}_t) + \Theta_t^* \Big( \frac{d\wtilde{\omega}_t}{dt} \Big) \\
    &= \cal{L}_{(\Theta_t^{-1})_* Y_t} (\Theta_t^* \wtilde{\omega}_t) + \Theta_t^* \Big( -2K\Ric (\wtilde{\omega}_t,J) \Big) \\
    &= \cal{L}_{K(\Theta_t^{-1})_* [\nabla_{\wtilde{h}_t} (\log |\Omega|_{\wtilde{\omega}_t})} \omega_t - 2K\Ric (\Theta_t^* \wtilde{\omega}_t, \Theta_t^* J) \\
    &= K\cal{L}_{\nabla_{h_t} (\log |\Omega_t|_{\omega_t})} \omega_t - 2K\Ric (\omega_t, J_t).
\end{aligned}
\end{equation*}

Using \eqref{eqn-Lie-deriv-KR-flow}, we get that
\begin{equation*}
    \frac{d\omega_t}{dt} = -K\cal{L}_{\nabla_{h_t} (\log |\Omega_t|_{\omega_t})} \omega_t,
\end{equation*}
which is just \eqref{eqn-ansatz-omega-coflow}.

We can also see using a similar calculation in \S \ref{subsect-soln-MA1/3} that
\begin{equation*}
\begin{aligned}
    \frac{d\Omega_t}{dt} = \frac{d}{dt} (\Theta_t^* \Omega) = K\cal{L}_{\nabla_{h_t} (\log |\Omega_t|_{\omega_t})} \Omega_t,
\end{aligned}
\end{equation*}
which is \eqref{eqn-ansatz-Omega-coflow}.

As such, we get that if we start the $G_2$-Laplacian coflow with initial data of the form
\begin{equation}
\label{eqn-varphi-coflow-initial-2}
    \varphi = 
    \begin{cases}
        - |\Omega|_\omega dr^1 \wedge dr^2 \wedge dr^3 + dr^1 \wedge |\Omega|_\omega \omega \\
        \qquad + dr^2 \wedge \Re \Big( \frac{1}{|\Omega|_\omega} \Omega \Big) + dr^3 \wedge \Im \Big( \frac{1}{|\Omega|_\omega} \Omega \Big) &\text{ on } M^7 = T^3 \times X^4, \\[5pt]
        \Re \Big( \frac{1}{|\Omega|_\omega} \Omega \Big) - dr \wedge |\Omega|_\omega \omega &\text{ on } M^7 = S^1 \times X^6.
    \end{cases}
\end{equation}
then the $G_2$-structures induced by the K\"{a}hler--Ricci flow and given by
\begin{equation}
\label{eqn-ansatz-varphi-coflow-3}
    \varphi_t = 
    \begin{cases}
        - \Theta_t^* (|\Omega|_{\wtilde{\omega}_t}) dr^1 \wedge dr^2 \wedge dr^3 + dr^1 \wedge \Theta_t^* (|\Omega|_{\wtilde{\omega}_t} \wtilde{\omega}_t) \\
        \qquad + dr^2 \wedge \Re \Big( \Theta_t^* \Big[ \frac{1}{|\Omega|_{\wtilde{\omega}_t}} \Omega \Big] \Big) + dr^3 \wedge \Im \Big( \Theta_t^* \Big[ \frac{1}{|\Omega|_{\wtilde{\omega}_t}} \Omega \Big] \Big) &\text{ on } M^7 = T^3 \times X^4, \\[5pt]
        \Re \Big( \Theta_t^* \Big[ \frac{1}{|\Omega|_{\wtilde{\omega}_t}} \Omega \Big] \Omega \Big) - dr \wedge \Theta_t^* (|\Omega|_{\wtilde{\omega}_t} \wtilde{\omega}_t) &\text{ on } M^7 = S^1 \times X^6.
    \end{cases}
\end{equation}
solve the $G_2$-Laplacian coflow equation \eqref{eqn-Laplacian-coflow-2} for all time $t \in [0,\infty)$. 

\section{Convergence of Solutions}

For this section, we use the convention where $C, C_k$ and $\lambda, \lambda_k$ denote generic positive constants that may change from line to line but do not depend on $t$.

\subsection{Estimates for Complex Monge--Amp\`ere Flows}
\label{sect-anal-prelims}

We briefly review analytic estimates regarding complex Monge--Amp\`ere flows. These estimates can be found in joint work by the first-named author and X.-W. Zhang \cite{PZ19}. These bounds build on Yau's estimates \cite{Yau78} and recover the estimates of Cao \cite{Cao85} for the K\"ahler--Ricci flow on Calabi--Yau manifolds as a special case. For follow-up work in the non-K\"ahler case, see \cite{Sm20}.

For a compact K\"{a}hler manifold $X$ with K\"{a}hler form $\omega$, the main result in \cite{PZ19} states that there exists a smooth solution $u_t$ for all time $t \in [0,\infty)$ to the parabolic complex Monge--Amp\`{e}re equation
\begin{equation}
\label{eqn-MA-flow}
    \frac{du_t}{dt} = H \Big( e^{-a} \frac{\det(\omega + i\partial \overline{\partial} u_t)}{\det \omega} \Big), \quad \omega + i \partial \overline{\partial} u_t > 0, \quad u_0=0
\end{equation}
where $H \: \bb{R}^+ \rarr \bb{R}$ is a smooth function with $H'>0$ and $a \in C^\infty(X,\bb{R})$ is a smooth function.

When $X$ is a Calabi--Yau manifold, we can use the holomorphic volume form $\Omega$ to define the function $a$. In particular for a constant $K > 0$, we can choose $H(\rho) = 6K\rho^{\frac{1}{3}}$ and $a = 2 \log |\Omega|_\omega$ and see that there exists a smooth solution $u_t$ to the MA$^\frac{1}{3}$ flow
\begin{equation}
\label{eqn-MA1/3-flow-2}
    \frac{du_t}{dt} = 6K \Big(e^{-2 \log |\Omega|_{\omega}} \frac{\det (\omega + i \partial \overline{\partial} u_t)}{\det \omega} \Big)^\frac{1}{3}
\end{equation}
for all time $t \in [0,\infty)$. On the other hand, by choosing $H(\rho) = 2K \log \rho$ and $a = 2 \log |\Omega|_\omega$, we obtain
\begin{equation}
\label{eqn-KR-flow-2}
    \frac{du_t}{dt} = 2K \Big( \log \frac{\det (\omega + i \partial \overline{\partial} u_t)}{\det \omega} - \log |\Omega|_\omega^2 \Big)
\end{equation}
which is the K\"ahler-Ricci flow: by setting $\wtilde{\omega}_t = \omega + i \partial \overline{\partial} u_t$, we see that $\frac{d}{dt} \wtilde{\omega}_t = - 2K{\rm Ric}(\wtilde{\omega}_t,J)$. 

\begin{thm} 
  \cite{PZ19} Let $(X,J,\omega,\Omega)$ be a compact complex manifold with K\"ahler form $\omega$, associated metric $h$, and holomorphic volume form $\Omega$. Consider the complex Monge--Amp\`ere flow \eqref{eqn-MA-flow} with $a(z)=2 \log |\Omega|_\omega$. The flow exists for all time $t >0$, and the metrics
  \[
\wtilde{\omega}_t = \omega + i \partial \overline{\partial} u_t
  \]
  are a family of K\"ahler metrics on $X$ in the K\"ahler class $[\omega]$ converging in each $C^k(X,h)$ norm to a limiting metric $\omega_\CY \in [\omega]$ solving
  \[
\omega_\CY^n = c_0 |\Omega|_\omega^2 \omega^n,
  \]
  for some constant $c_0>0$ which can be determined by integration. It follows that the limit $\omega_\CY$ is the unique K\"ahler Ricci-flat metric in the K\"ahler class of $\omega$.

  The evolving metrics satisfy the uniform estimates
\begin{equation}
\label{eqn-unif-est-metrics-MA1/3-flow}
    C^{-1}h \leq \wtilde{h}_t \leq Ch,
  \end{equation}
  and
\begin{equation}
\label{eqn-unif-est-nabla-omega-MA1/3-flow}
    \Big|\nabla_{h}^k \wtilde{\omega}_t \Big|_h \leq C_k,
  \end{equation}
  where as usual $\wtilde{h}_t(JY,Z)= \wtilde{\omega}_t(Y,Z)$.
\end{thm}

We will also need exponential convergence of the flow. This is the estimates
\begin{equation}
\label{eqn-unif-est-nabla-omega-flow}
    \Big| \frac{d}{dt} \nabla_{h}^k \wtilde{\omega}_t \Big|_{h} \leq C_k e^{-\lambda_k t},
\end{equation}
for each $k \in \{0,1,2,\ldots\}$.

For the K\"ahler-Ricci flow, exponential convergence was obtained by Phong--Sturm \cite{PS06}. In our setup, this follows from estimate (4.104) and Proposition 3 in \cite{PZ19}, which reads
\[
\| \wtilde{u}_t \|_{C^k(X,h)} \leq C_k, \quad \Big|\frac{d}{dt} {\wtilde{u}}_t \Big| \leq C e^{-\lambda t},
\]
where $\wtilde{u}_t = u_t - \frac{1}{V} \int_X u_t \omega^n$ and $V = \int_X \omega^n$. To get decay of higher derivatives of $\wtilde{u}_t$, we integrate by parts. For example,
\[
\int_X \Big| \frac{d}{dt} \nabla {\wtilde{u}}_t \Big|^2 \omega^n = \int_X \bigg| \frac{d}{dt} {\wtilde{u}}_t \frac{d}{dt}\Delta {u}_t \,  \bigg| \omega^n \leq C \| \frac{d}{dt} {\wtilde{u}}_t \|_{L^\infty(X,h)}  \| \Delta H \|_{L^\infty(X,h)} \leq C e^{-\lambda t}
\]
since $|\Delta H| \leq \|u_t \|_{C^4(x,h)} \leq C$. Similarly, for any integer $k \geq 2$ then
\[
\int_X \Big| \frac{d}{dt} \nabla^k {\wtilde{u}}_t \Big|^2 \omega^n  \leq C \| \frac{d}{dt} {\wtilde{u}}_t \|_{L^\infty(X,h)} \| \nabla^{k+1} H \|_{L^\infty(X,h)} \leq C_k e^{-\lambda_k t}.
\]
The Sobolev embedding theorem implies that
\begin{equation*}
    \| \frac{d}{dt} \nabla^k {\wtilde{u}}_t \|_{L^\infty(X,h)} \leq C_k e^{- \lambda_k t}.
\end{equation*}

Since $\wtilde{\omega}_t = \omega + i \partial \overline{\partial} \wtilde{u}_t$, we obtain \eqref{eqn-unif-est-nabla-omega-flow}.

\subsection{Limit of Diffeomorphisms}
\label{sect-convergence}

Returning to $G_2$ geometry, it remains to show that our solutions $\varphi_t$ converge to a torsion-free $G_2$-structures $\varphi_\infty$ in both the flow and coflow cases. As both solutions involve pullbacks via a family of diffeomorphisms and their respective geometric flows have similar properties and estimates, we can tackle both cases simultaneously.

Recall that in either case:

$\bullet$ The flow is solved by the pair $(\omega_t,\Omega_t) = (\Theta_t^* \wtilde{\omega}_t, \Theta_t^* \Omega)$.

$\bullet$ The time-dependent family of K\"{a}hler triples $(\wtilde{\omega}_t, J, \wtilde{h}_t)$ on $X^{2n}$ come from a complex Monge--Amp\`ere flow (either MA$^\frac{1}{3}$ or K\"ahler-Ricci flow).

$\bullet$ $\wtilde{\omega}_t$ satisfies the estimates \S \ref{sect-anal-prelims} and converges to the unique K\"{a}hler Ricci-flat metric $\omega_\CY$ in the K\"{a}hler class $[\omega]$ in each $C^k(X^{2n},h)$-norm where $h$ is the original background metric.

$\bullet$ The diffeomorphisms solve $\frac{d}{dt} \Theta_t = Y_t$, where the time-dependent vector fields $Y_t$  were defined by \eqref{eqn-Y_t-flow} in the Laplacian flow case and by \eqref{eqn-Y_t-coflow} in the coflow case.

To prove convergence of $(\omega_t,\Omega_t)$, we use a method similar to that in \cite{LW19} to show that the maps $\Theta_t$ converge to a diffeomorphism $\Theta_\infty$.

Since along both the MA$^{\frac{1}{3}}$ and K\"{a}hler--Ricci flows the metric $\wtilde{\omega}_t$ converges to $\omega_\CY$ exponentially fast as $t \rarr \infty$ in each $C^k(X^{2n},h)$-norm and $|\Omega|_{\omega_\CY}$ is constant, it follows that $Y_t \rarr 0$ also converges exponentially fast as $t \rarr \infty$. Indeed
\[
|\nabla_{h}^k Y_t|_{h} = - \int_t^\infty \frac{d}{ds} |\nabla_{h}^k Y_s|_{h} ds \leq C \int_t^\infty e^{-\lambda s} ds,
\]
by \eqref{eqn-unif-est-nabla-omega-flow}, and for each $k \geq 0$ we have the estimates
\begin{equation}
\label{eqn-unif-est-nabla-Y}
    |\nabla_{h}^k Y_t|_{h} \leq C_k e^{-\lambda_k t}.
  \end{equation}

We now consider the family of diffeomorphisms $\Theta_t$. For every point $p \in X^{2n}$ and $t_1,t_2 \geq 0$, we have that
\begin{equation}
\label{eqn-Theta_t-path}
    [t_1,t_2] \rarr X^{2n} \: t \mapsto \Theta_t(p)
\end{equation}
defines a smooth path from $\Theta_{t_1}(p)$ to $\Theta_{t_2}(p)$. By \eqref{eqn-unif-est-nabla-Y} we can then see that
\begin{equation*}
\begin{aligned}
    d_{h}(\Theta_{t_1}(p), \Theta_{t_2}(p)) &\leq \int_{t_1}^{t_2} \Big| \frac{d}{dt} \Theta_t(p) \Big|_h dt \\
    &\leq \int_{t_1}^{t_2} |Y_t|_{h} dt \\
    &\leq C \int_{t_1}^{t_2} e^{-\lambda t} dt.
\end{aligned}
\end{equation*}
It follows that the maps $\Theta_t$ converge uniformly with respect to the metric $h$. Similarly, by \eqref{eqn-unif-est-nabla-Y} and the uniform estimates \eqref{eqn-unif-est-metrics-MA1/3-flow} and \eqref{eqn-unif-est-nabla-omega-MA1/3-flow} we have that $\Theta_t$ converges in each $C^k(X^{2n},h)$-norm. Thus we have that $\Theta_t$ converges to some limit map $\Theta_\infty \: X^{2n} \rarr X^{2n}$ as $t \rarr \infty$ in each $C^k(X^{2n},h)$-norm.

Next, we show that the pullback $\Theta_\infty^*$ is not degenerate. For this, we estimate
\begin{equation*}
\begin{aligned}
    \Big| \frac{d}{dt} \log \Big( \frac{\Omega_t \wedge \overline{\Omega_t}}{\Omega \wedge \overline{\Omega}} \Big) \Big| &= \Big| \frac{d}{dt} \Big( \log \frac{\Theta_t^* (\Omega \wedge \overline{\Omega})}{\Omega \wedge \overline{\Omega}} \Big) \Big| \\
    &= \Big| \frac{1}{\Theta_t^* (\Omega \wedge \overline{\Omega})} \frac{d}{dt} \Big( \Theta_t^* (\Omega \wedge \overline{\Omega}) \Big) \Big| \\
    &= \Big| \Theta_t^* \Big( \frac{\cal{L}_{Y_t}(\Omega \wedge \overline{\Omega})}{\Omega \wedge \overline{\Omega}}    \Big) \Big| \\
    &\leq \sup_X \Big| \Big( \frac{\cal{L}_{Y_t} (|\Omega|^2_\omega \vol_{2n}) }{|\Omega|^2_\omega \vol_{2n}}    \Big) \Big| \\
    &\leq \frac{| Y_t (|\Omega|^2_\omega)|}{|\Omega|^2_\omega} + \Big| \frac{d (Y_t \hook \vol_{2n})}{\vol_{2n}} \Big| \leq C e^{-\lambda t}
  \end{aligned}
\end{equation*}
by \eqref{eqn-unif-est-nabla-Y}. As such
\[
    \Big| \log \Big( \frac{\Omega_t \wedge \overline{\Omega_t}}{\Omega \wedge \overline{\Omega}} \Big) \Big| \leq \int_0^t \Big| \frac{d}{ds} \log \Big( \frac{\Omega_s \wedge \overline{\Omega_s}}{\Omega \wedge \overline{\Omega}} \Big) \Big| ds \leq C \int_0^t e^{-\lambda s} ds \leq C
    \]
is bounded independently of $t$. So it follows that
\begin{equation}
\label{eqn-unif-est-Omega_t-wedge-Omega_t-bar}
    C^{-1} \Omega \wedge \overline{\Omega} \leq \Theta_t^* (\Omega \wedge \overline{\Omega}) \leq C \Omega \wedge \overline{\Omega},
\end{equation}
and hence the pullback is uniformly non-degenerate. We see that $\det (\Theta_t)_*$ is bounded independently of $t$ and this estimate can be passed to the limit map $\Theta_\infty$.

By the inverse function theorem, $\Theta_\infty$ is a local diffeomorphism. Since $\Theta_0 = \id_{X^{2n}}$ is the identity map and each $\Theta_t$ is a diffeomorphism which is isotopic to the identity map, we have that $\Theta_\infty$ is a surjective local diffeomorphism homotopic to the identity. Further, $X^{2n}$ is compact, and so $\Theta_\infty$ is a covering map. Lastly, $\Theta_\infty$ is homotopic to the identity map and so it must have degree $1$ and is thus injective. We conclude that $\Theta_\infty$ is indeed a diffeomorphism.

Finally, we have that $\Theta_t \rarr \Theta_\infty$ and $\wtilde{\omega}_t \rarr \omega_\CY$ with respect to the background metric $h$. It follows that $\Theta_t^* \wtilde{\omega}_t \rarr \Theta_\infty^* \omega_\CY$ and $\Theta_t^* \Omega \rarr \Theta_\infty^* \Omega$ also with respect to the background metric $h$. As such, we have that
\begin{equation}
\label{eqn-varphi-flow-stationary-2}
    \varphi_t \rarr \varphi_{\infty} = 
    \begin{cases}
        - dr^1 \wedge dr^2 \wedge dr^3 + dr^1 \wedge \Theta_\infty^* \omega_{{\CY}} \\
        \qquad + dr^2 \wedge \Re (\Theta_\infty^*\Omega) + dr^3 \wedge \Im (\Theta_\infty^*\Omega) &\text{ on } T^3 \times X^4 \\[5pt]
        \Re( \Theta_\infty^*\Omega) - dr \wedge \Theta_\infty^* \omega_{{\CY}} &\text{ on } S^1 \times X^6,
    \end{cases}
\end{equation}
in the Laplacian flow case
\begin{equation}
\label{eqn-varphi-coflow-stationary-2}
    \varphi_t \rarr \varphi_{\infty} = 
    \begin{cases}
        - \Theta_\infty^* (|\Omega|_{\omega_\CY}) dr^1 \wedge dr^2 \wedge dr^3 + dr^1 \wedge \Theta_\infty^* (|\Omega|_{\omega_\CY} \omega_\CY) \\
        \qquad + dr^2 \wedge \Re \Big( \Theta_\infty^* \Big[\frac{1}{|\Omega|_{\omega_{\CY}}} \Omega\Big] \Big) + dr^3 \wedge \Im \Big( \Theta_\infty^* \Big[\frac{1}{|\Omega|_{\omega_{\CY}}} \Omega\Big] \Big) &\text{ on } T^3 \times X^4, \\[5pt]
        \Re \Big( \Theta_\infty^* \Big[\frac{1}{|\Omega|_{\omega_{\CY}}} \Omega\Big] \Big) -  dr \wedge \Theta_\infty^* (|\Omega|_{\omega_{\CY}} \omega_{{\CY}}) &\text{ on } S^1 \times X^6,
    \end{cases}
\end{equation}
in the Laplacian coflow case.

As noted earlier, $|\Omega|_{\omega_\CY}$ is constant, so by Lemmas \ref{lem-torsion-flow-dim-2}, \ref{lem-torsion-flow-dim-3}, \ref{lem-torsion-coflow-dim-2}, and \ref{lem-torsion-coflow-dim-3}, $\varphi_\infty$ defines a torsion-free $G_2$-structure in all cases.


\end{document}